\documentclass[reqno]{amsart}
\usepackage{amssymb,latexsym,color,}
\usepackage{amsmath}
\usepackage{amsthm}
\usepackage{graphicx}
\usepackage{titletoc}

\numberwithin{equation}{section}
\newtheorem{theorem}{Theorem}[section]
\newtheorem{proposition}[theorem]{Proposition}
\newtheorem{lemma}[theorem]{Lemma}
\newtheorem{remark}[theorem]{Remark}
\newtheorem{corollary}[theorem]{Corollary}

\theoremstyle{definition}

\def\XXint#1#2#3{{\setbox0=\hbox{$#1{#2#3}{\int}$}
     \vcenter{\hbox{$#2#3$}}\kern-.5\wd0}}

\newcommand{\R}{\mathbb{R}}

\begin{document}
\title[Degree formula, the simple singular source]{
Degree counting for Toda system with simple singularity : one point blow up}
\author{Youngae Lee}
\address{Youngae ~Lee,~Center for Advanced Study in Theoretical Science,National Taiwan University, No.1, Sec. 4, Roosevelt Road, Taipei 106, Taiwan}
\email{youngaelee0531@gmail.com}
\author{ Chang-shou Lin}
\address{ Chang-shou ~Lin,~Taida Institute for Mathematical Sciences, Center for Advanced Study in
Theoretical Sciences, National Taiwan University, Taipei 106, Taiwan}
\email{cslin@math.ntu.edu.tw}
\author{Wen Yang}
\address{ Wen ~Yang,~Center for Advanced Study in Theoretical Sciences (CASTS), National Taiwan University, Taipei 10617, Taiwan}
\email{math.yangwen@gmail.com}
\author{Lei Zhang}
\address{ Lei ~Zhang,~Department of Mathematics, University of Florida, 358 Little Hall P.O.Box 118105, Gainesville FL 32611-8105}
\email{leizhang@ufl.edu}
\date{}
\begin{abstract}
In this paper, we study the degree counting formula of the rank two Toda system with simple singular source when $\rho_1\in(0,4\pi)\cup(4\pi,8\pi)$ and $\rho_2\notin 4\pi\mathbb{N}.$ The key step is to derive the degree formula of the shadow system, which arises from the bubbling solutions as $\rho_1$ tends to $4\pi$. In order to compute the topological degree of the shadow system, we need to find some suitable deformation. During this deformation, we shall deal with \textit{new} difficulty arising from the new phenomena: blow up does not necessarily imply concentration of mass. This phenomena occurs due to the  collapsing of singularities.  This is a continuation of the previous work \cite{llwy}.
\end{abstract}

\keywords{Toda system; topological degree; bubbling solutions; shadow system}
\maketitle

\section{Introduction}

\subsection{Shadow system}

Let $(M,g)$ be a compact Riemann surface with volume $1$  and $\Delta$ is the corresponding Laplace-Beltrami operator. In this paper, we are devoted to compute the  Leray-Schauder topological degree of the Toda system of rank $2$ (see $(1.13)$ or $(1.14)$ below).  Our strategy is to reduce this degree counting problem to a single equation, the so-called \textit{shadow} system of the corresponding Toda system:
\begin{align}
\label{1.1}
\left\{\begin{array}{l}
\Delta w+2\rho_2\left(\frac{h_2e^{w+4\pi K_{21} G(x,Q)}}{\int_Mh_2e^{w+4\pi K_{21} G(x,Q)}}-1\right)=0,\\
\nabla\big(\log h_1e^{\frac{K_{12}}{2}w}\big)\mid_{x=Q}=0,~\mathrm{and}~Q\notin S_1,
\end{array}\right.
\end{align}
where $i=1,2$, $h_i(x)=h_i^*(x)e^{-4\pi\sum_{p\in S_i}\alpha_{p,i}G(x,p)},$ $h_i^*>0$ in $M$ and $S_i$ is a finite set  in $M$.
Here $G(x,p)$ is the Green function  on $M$  satisfying
\begin{equation}\label{defofgreenfunction}
-\Delta G(x,p)=\delta_p-1~\mathrm{in}~M,~\ \ \mathrm{and}~\int_MG(x,p)=0.
\end{equation}

Throughout this paper, $\alpha_{p,i}$ is a \textit{positive} integer for $p\in S_i$, $i=1,2$.
Here $\textbf{K}=(K_{ij})_{2\times 2}$ is the Cartan matrix of rank $2$. See the subsection 1.2 below for the forms of $\textbf{K}$.
\medskip

To well-define the topological degree of \eqref{1.1}, we shall first prove compactness of the solution of  \eqref{1.1} in some function space. Here the function space is $\mathring{H}^1(M)\times [M\setminus S_1]$  (see \eqref{defofH1} for the definition of $\mathring{H}^1(M)$) and \eqref{1.1} is the zero set of the nonlinear map
 \begin{align*}
(w,Q)\overset{\Phi}{\longrightarrow} \left(
\Delta w+2\rho_2\left(\frac{h_2e^{w+4\pi K_{21} G(x,Q)}}{\int_Mh_2e^{w+4\pi K_{21} G(x,Q)}}-1\right),
\nabla(\log h_1e^{\frac{K_{12}}{2}w})(Q)\right).
\end{align*} Then we have the following compactness theorem.
\begin{theorem}\label{compactnessthm}   Suppose $\alpha_{p,i}\in\{1,2\}$ and $\rho_2\notin4\pi\mathbb{N}$, where $\mathbb{N}$ is the set of positive integers. Then there are constants $C>0$ and $\delta>0$ such that
for any solution $(w,Q)$ of \eqref{1.1},
\[\|w\|_{C^1(M)}\le C \  \textrm{and}\ \textrm{dist}(Q,S_1)\ge\delta>0.\]
\end{theorem}
Suppose that  any solution of $(w,Q)$ of \eqref{1.1} is non-degenerate.   Then the Morse index $M(w,Q)$ of $\Phi$ at $(w,Q)$ is the number of negative eigenvalues of the linearized equation at $(w,Q)$, and the topological degree is defined by \[\sum_{(w,Q)\ \textrm{is solution of}\ \eqref{1.1}}(-1)^{M(w,Q)}.\]
The  Leray-Schauder topological degree is well-defined for \eqref{1.1}, even though the non-degeneration of \eqref{1.1} is violated. We refer the readers to \cite{n}. Without the second equation of \eqref{1.1} (which will be referred as \textit{the balance condition} throughout this paper), the nonlinear PDE itself is called a mean field equation and the degree counting  formula has been proved in a series of papers by Chen and Lin (see \cite{cl1,cl2,cl3,cl4}). Let us briefly recall  Chen and Lin's degree counting formula: Consider the mean field equation,
\begin{equation}
\label{1.7}
\Delta u^*+\rho\left(\frac{h^*e^{u^*}}{\int_Mh^*e^{u^*}dv_g}-1\right)=4\pi\sum_{p\in S_0}\alpha_p(\delta_p-1),
\end{equation}where $\rho$ is {a positive} parameter, $\alpha_p>-1$ for every $p\in S_0$ and $S_0$ is a subset of $M.$\\
We set
$u^*(x)=u(x)-4\pi\sum_{p\in S_0}\alpha_{p}G(x,p).$ Then (\ref{1.7}) can be reduced to the equation without singular source,
\begin{equation}
\label{1.8}
\Delta u+\rho\Big(\frac{\overline{h}e^{u}}{\int_M\overline{h}e^{u}dv_g}-1\Big)=0,
\end{equation}
where $\overline{h}(x)=h^*(x)e^{-\sum_{p\in S_0}4\pi\alpha_pG(x,p)}\geq0$ in $M$ and $\overline{h}=0$ if and only if $x\in S_0$. Note that (\ref{1.8}) is
invariant by adding a constant to the solutions. Therefore, we can always consider the equation (\ref{1.8}) in the following function space:
\begin{align}\label{defofH1}
\mathring{H}^1(M)=\Big\{u\in H^1(M)\ \Big|\ \int_Mudv_g=0\Big\}.
\end{align}
For equation (\ref{1.8}), we introduce the set of critical parameters
\begin{align*}
\Sigma:&=\big\{8N\pi+\Sigma_{p\in A}8\pi(1+\alpha_{p})\mid N\in\mathbb{N}\cup\{0\},\ A\subseteq S_0\big\}\setminus\{0\}\\
&=\{8\pi \mathfrak{a}_{k}\mid \mathfrak{a}_{1}\le \mathfrak{a}_{2}\le \cdots\}.
\end{align*}
Through a series of work by Brezis-Merle \cite{bm}, Li-Shafrir \cite{ls} and Bartoluci-Tarantello \cite{bt1},  a priori bound  for the solutions to \eqref{1.8} was established:

\bigskip

\noindent {\bf{Theorem A}.} (\cite{bt1,bm,ls}) Let $\rho\notin\Sigma$, then all the solutions of (\ref{1.8}) are uniformly bounded.

\bigskip

Let \[T_{\rho}u=\rho\Delta^{-1}\left(\frac{\overline{h}e^u}{\int_M\overline{h}e^udv_g}-1\right).\] By Theorem A, the Leray-Schauder degree
$$d_{\rho}:\mathrm{deg}(I+T_{\rho},B_R,0)$$
is well defined for $\rho\notin\Sigma,$ where $B_R=\{u\in\mathring{H}^1(M)\mid\|u\|_{H^1(M)}\leq  R\}$. Since $d_{\rho}$ is a homotopic invariant, $d_{\rho}$ is a constant for $\rho\in(8\mathfrak{a}_j\pi,8\mathfrak{a}_{j+1}\pi)$ (for convenience, we set $\mathfrak{a}_0=0$), which is denoted by $d_j,j=0,1,\cdots,$ obviously $d_0=1$.
To represent $d_j$, it is better to introduce the generating function
$$g^{(1)}(x)=\sum_{j=0}^{\infty}d_jx^j.$$

In \cite{l}, Li pointed out that the degree formula should depend only on the topology of $M$ for the case without singularities. In \cite{cl2,cl4}, Chen and Lin obtained the degree counting formula  {for general cases as }stated below.

\vspace{0.5cm}
\noindent {\bf{Theorem B}}.   {\em Let $d_{j}$ be the Leray-Schauder degree for (\ref{1.7}) with $\rho\in(8\mathfrak{a}_j\pi,8\mathfrak{a}_{j+1}\pi)$. Then the generating function $g^{(1)}(x)$ is determined by,
\begin{align*}
g^{(1)}(x)=(1-x)^{\chi(M)-|S_0|-1}\prod_{p\in S_0}(1-x^{1+\alpha_p}),
\end{align*}
where $\chi(M)$ is the Euler characteristic of $M$. Consequently if $\chi(M)\leq 0$ and $\alpha_p\in\mathbb{N}$ for any $p\in S_0$, then
\begin{equation}
\label{1.9}
g^{(1)}(x)=(1+x+\cdots)^{1-\chi(M)}\prod_{p\in S_0}(1+x+\cdots+x^{\alpha_p}),
\end{equation}
and $d_j>0$ for $j\geq0.$}

\vspace{0.5cm}

Once the a priori bound is established by Theorem \ref{compactnessthm}, we can define the Leray-Schauder degree for (\ref{1.1}) when $\rho_2\in(4j\pi,4(j+1)\pi)$.  We denote it by $d_j^S$.
Our first main result is to obtain the generating function for $d_j^S$:
\begin{theorem}
\label{th1.2}
Let $d_j^S$ be the Leray-Schauder degree for (\ref{1.1}) when $\rho_2\in(4j\pi,4(j+1)\pi)$. Suppose $\alpha_{p,1},\alpha_{p,2}\in\{1,2\}.$ Then the generating function
$$g_s(x)=\sum_{j=0}^\infty d_j^Sx^j$$
is determined by
\begin{align}
\label{1.15}
g_s(x)=~&(1-x)^{\chi(M)-1}\Big[\left(\chi(M)-|S_2\cup S_1|\right)(1+\cdots+x^{-K_{21}})\prod_{p\in S_2}(1+\cdots+x^{\alpha_{p,2}})\nonumber\\
&+\sum_{p\in S_2\setminus S_1}(1+x+\cdots+x^{\alpha_{p,2}-K_{21}})\prod_{q\in S_2\setminus\{p\}}(1+x+\cdots+x^{\alpha_{q,2}})\Big].
\end{align}
\end{theorem}
\medskip

\noindent We note that all the coefficients of the following polynomial is nonnegative:
\begin{align*}
(1+\cdots+x^{-K_{21}})(1+\cdots+x^{\alpha_{p,2}})-(1+x+\cdots+x^{\alpha_{p,2}-K_{21}}).
\end{align*}
As a consequence, if   either $\chi(M)<0$ or $\chi(M)=0$, $S_1\cup S_2\neq\emptyset$, we can show all the coefficients of $g_s(x)$ are negative. Thus, we have the following corollary.

\begin{corollary}
\label{cr1.1}
Suppose the assumptions of Theorem \ref{th1.2} hold. If either $\chi(M)<0$ or $\chi(M)=0$, $S_1\cup S_2\neq\emptyset$,  then the shadow system (\ref{1.1}) has a solution for $\rho_2\notin4\pi\mathbb{N}.$
\end{corollary}

When $S_i\neq\emptyset, i=1,2,$ it is rather nontrivial to calculate the topological degree of the shadow system (\ref{1.1}). Let us briefly discuss our approach for the calculation.

If (\ref{1.1}) has no singularity, i.e., $S_1\cup S_2=\emptyset,$ then both $h_1,h_2$ are positive smooth functions on $M$. We consider $(w,Q)$ are defined in $\mathring{H}^1(M)\times M$ and $(w,Q)$ is the zero of the nonlinear map
\begin{align*}
(w,Q)\overset{\Phi}{\longrightarrow} \left(
\Delta w+2\rho_2\left(\frac{h_2e^{w+4\pi K_{21} G(x,Q)}}{\int_Mh_2e^{w+4\pi K_{21} G(x,Q)}}-1\right),
\nabla(\log h_1e^{\frac{K_{12}}{2}w})(Q)\right).
\end{align*}
It is easy to see that the compactness of $\Phi^{-1}(0)$ is equivalent to the a priori estimate of $\|w\|_{C^1(M)}\leq C$ for any solution $(w,Q)\in\Phi^{-1}(0)$. In order to compute the topological degree, we introduce the deformation $\Phi_t$ of $\Phi$
\begin{align*}
~\quad\quad\quad\quad\quad\quad\quad\left\{\begin{array}{l}
\Delta  {w_t}+2\rho_2\left(\frac{h_2e^{ {w_t}+4\pi K_{21} G(x, {Q_t})}}{\int_Mh_2e^{ {w_t}+4\pi K_{21} G(x, {Q_t})}}-1\right)=0,\\
\nabla\big(\log h_1e^{\frac{t}{2}K_{12} {w_t}}\big)\mid_{x= {Q_t}}=0.
\end{array}\right.~\quad\quad\quad\quad\quad(1.8)_t
\end{align*}
Since $\rho_2\notin4\pi\mathbb{N}$ and $S_1\cup S_2=\emptyset,$ the compactness of $(1.8)_t$ for $t\in[0,1]$ even holds without the balance condition at $Q_t$, this is a simple consequence of Theorem A. Thus, the degree of the shadow system (\ref{1.1}) with $S_1\cup S_2=\emptyset$ is equal to the degree of the system $(1.8)_0$ which is a de-coupled system, and the degree for $(1.8)_0$ follows from Theorem B.
\medskip

However, when $S_1\cup S_2\neq\emptyset$, it becomes much harder. There are two cases for (\ref{1.1}): one is $Q\in S_2\setminus S_1$ and the other is $Q\notin S_2\cup S_1.$ For the first case, the degree of the system can be calculated as before. But for the later, the domain of $\Phi$ is $\mathring{H}^1(M)\times[M\setminus\{S_1\cup S_2\}].$ We note  that there is no information for  $S_2\setminus S_1$ in the balance condition.  The problem is that there might be a sequence of  solutions $(w_k,Q_k)$ of \eqref{1.1} such that $Q_k\notin S_1\cup S_2$ and $Q_k\to Q_0\in S_2$. This is the phenomena of collapsing singularities. There are two cases to be discussed:
\begin{itemize}
\item[(i)] $w_k$ blows up, or
\item[(ii)] $w_k$ does not blow up.
\end{itemize}

For the case (i), we consider  a general class of the mean field equation with collapsing singularities:
 \begin{equation*}
\quad\quad\quad\quad\quad\quad\ \Delta \hat{u}_k+2\rho_2    \hat{h} e^{\hat{u}_k}  =4\pi\sum_{p_{k_j}\in\hat{S}_k}\beta_{j} \delta_{p_{k_j}}\ \textrm{in}\ B_1(0),\quad\quad\quad\quad\quad\quad\quad\quad(1.9)
\end{equation*}  where $ \hat{h}>0$,  $|\hat{S}_k|$ is independent of $k$,    $\lim_{k\to+\infty}p_{k_j}=0$ for all $p_{k_j}\in \hat{S}_k$,  $p_{k_i}\neq p_{k_j}$ if $i\neq j$,  and  $\beta_{j}\in\mathbb{N}$.
To the best of our knowledge, there have been no  available estimates for $(1.9)$. There might be a new phenomena  such that blow up does not necessarily imply concentration of mass. We refer the readers to \cite{lt}.
Our second main result is to show that  the local mass is an even positive integer, despite the existence of collapsing singularities.
\begin{theorem}\label{thmeven}Let $\hat{u}_k$ be a solution of $(1.9)$.  We assume that  $0$   is the only blow up point, $\hat{u}_k$ has the bounded oscillation on $\partial B_1(0)$ and finite mass (see also \eqref{assumption}).  Then the local mass $\sigma_0$ satisfies
\[\sigma_0:=\lim_{\delta\to0}\lim_{k\to+\infty}\frac{1}{2\pi} \int_{B_{\delta}(0)}\rho_2\hat{h} e^{\hat{u}_k}\in2\mathbb{N}.\]
\end{theorem}
Even though the blow up case with collapsing singularities  can be excluded by Theorem \ref{thmeven}, we still could not  prove the compactness of the solutions of $\Phi$ by just showing $w$ is uniformly bounded. Indeed, it is possible to get a sequence of solutions $(w_k,Q_k)$ to (\ref{1.1}) such that $Q_k\rightarrow Q_0\in S_2\setminus S_1$ and $w_k$ is uniformly bounded, that is, the case (ii). This is the reason why we could not use the deformation $(1.8)_t$ to calculate the degree for the case $S_1\cup S_2\neq\emptyset$. Instead, we introduce a new term which contains the information of $S_2\setminus S_1$ in the deformation. We set
\begin{align*}
~\quad\quad\quad\left\{\begin{array}{l}
\Delta  {w_t}+2\rho_2\left(\frac{h_2e^{ {w_t}+4\pi K_{21} G(x, {Q_t})}}{\int_Mh_2e^{ {w_t}+4\pi K_{21} G(x, {Q_t})}}-1\right)=0,\\
\nabla\left(\log h_1e^{\frac{t}{2}K_{12} {w_t}}-4\pi(1-t)\sum_{ {q}\in S_2\setminus S_1}G(x, {q})\right)\mid_{x= {Q_t}}=0.
\end{array}\right.~\quad\quad(1.10)_t
\end{align*}
Obviously, the corresponding function space is $\mathring{H}^1(M)\times [M\setminus(S_1\cup S_2)]$ for any $t\in[0,1)$.
However,  on one hand, we note that when $t\rightarrow1^-,$  the system $(1.10)_t$ does not converge to the original system (\ref{1.1}) as $t\rightarrow1^-,$
because
\begin{align*}
~\ \quad\quad\quad\quad\quad\quad\quad\lim_{t\rightarrow1^-} {\sum_{q\in S_2\setminus S_1}4\pi(1-t)\nabla G(Q_t,q)}\neq0,
\quad\quad\quad\quad\quad\quad\quad\quad\quad(1.11)
\end{align*}
if $Q_t$ is collapsing with some element $Q_0\in S_2\setminus S_1$, and $w_t$ converges as $t\rightarrow1^-$. So, we have to find out what is the difference between (\ref{1.1}) and $(1.10)_t$ when $|1-t|\ll1.$ On the other hand, when $t=0$, system $(1.10)_t$ becomes the following decoupled system.
\begin{align*}
~\quad\quad\quad\quad\quad\quad\quad
\left\{\begin{array}{l}
\Delta w+2\rho_2\left(\frac{h_2e^{w+4\pi K_{21} G(x,Q)}}{\int_Mh_2e^{w+4\pi K_{21} G(x,Q)}}-1\right)=0,\\
\nabla\left(\log h_1-4\pi {\sum_{q\in S_2\setminus S_1}G(x,q)}\right)\mid_{x=Q}=0.
\end{array}\right.~\quad\quad\quad\quad\quad(1.12)
\end{align*}
The system (1.12) is a de-coupled system and we can easily compute the degree, combined with what we get from the differences between (\ref{1.1}) and $(1.10)_t$. Then we can derive the degree formula of (\ref{1.1}).
\medskip

\subsection{Toda system}
Our second purpose is to apply   the degree formula  \eqref{1.15} of the shadow system to compute the topological degree of the Toda system corresponding to a semi-simple Lie algebra of rank $2$. In this paper, we only consider the case of rank two. There are
only three types of rank two: $\textbf{A}_2$, $\textbf{B}_2=\textbf{C}_2$ and $\textbf{G}_2$ and their Cartan matrix $\textbf{K}=(K_{ij})$ is $\left(\begin{array}{ll}~2&-1\\-1&~2\end{array}\right)$,
$\left(\begin{array}{ll}~2&-1\\-2&~2\end{array}\right),$ $\left(\begin{array}{ll}~2&-1\\-3&~2\end{array}\right)$ respectively.

The corresponding Toda system (of mean field type) is
\begin{equation*}
\quad\quad\ \Delta u^*_i+\sum_{j=1}^2K_{ij}\rho_j\left(\frac{h_j^*e^{u_j^*}}{\int_Mh_j^*e^{u_j^*}}-1\right)=4\pi\sum_{p\in S_i}\alpha_{p,i}(\delta_p-1), \ i=1,2,
\quad\quad  \ (1.13)\end{equation*}where   $h_i^*$ are positive smooth functions on $M$, $\rho_i$ are  {positive} constants, $\alpha_{p,i}\in\mathbb{N}$ for every $p\in S_i$ and $S_i$ is a subset of $M$, $\delta_p$ is the Dirac measure at $p\in M$.

For $(1.13)$, conventionally we let
$u_i^*(x)=u_i(x)-4\pi\sum_{p\in S_i}\alpha_{p,i} G(x,p)$.
Then $u_i$ satisfies
\begin{equation*}
\quad\quad\quad\quad
\Delta u_i+\sum_{j=1}^2K_{ij}\rho_j\left(\frac{h_je^{u_j}}{\int_Mh_je^{u_j}}-1\right)=0\ \textrm{in}\ M, \ i=1,2,
\quad\quad\quad\quad\quad\quad \ (1.14)\end{equation*}
where \begin{equation*}
\quad\quad\quad\quad\quad\quad\quad\quad h_i(x)=h_i^*(x)e^{-4\pi\sum_{p\in S_i}\alpha_{p,i}G(x,p)},\ \ h_i^*>0.\quad\quad\quad\quad\quad\quad\ \ (1.15)\end{equation*}
%We note
%\medskip

%\noindent (H): $h_i(x)\geq0$ in $M$ and $h_i(x)=0$ if and only if $x\in S_i,~i=1,2$. Near each $p\in S_i$, $h_i(x)$ has the form in local coordinate:
%$$h_i(x)=h_{i,p}(x)|x-p|^{2\alpha_{p,i}}~\mathrm{and}~h_{i,p}(x)>0~\mathrm{for}~|x-p|\ll1,~\forall p\in S_i,~i=1,2.$$

%\noindent {\bf Remark}: Throughout this article, we always consider (1.14) with $(K_{ij})_{2\times 2}$ is Cartan matrix of rank 2 and (H) holds.
%\medskip

\noindent Clearly, the equation $(1.14)$ remains the same if each component $u_i$  is replaced by $u_i+c_i$, where $c_i$ is a constant.
Thus, we  assume that $u_i\in \mathring{H}^1(M)$.

It is known that equation $(1.13)$ is closely related to the classical Pl$\ddot{u}$cker formula for the holomorphic curves in projective space. Let $f$ be a holomorphic curve from a simple domain $D$ in $\mathbb{C}$ into $\mathbb{CP}^n$. Lift locally $f$ to $\mathbb{C}^{n+1}$ and denote the lift by $\nu(z)=\left[\nu_0(z),\nu_1(z),\cdots,\nu_n(z)\right]$. The $k$th associated curve of $f$ is defined by
$$f_k:D\rightarrow G(k,n+1)\subset\mathbb{CP}(\Lambda^{k}\mathbb{C}^{n+1}),~
f_k(z)=\left[\nu(z)\wedge\nu'(z)\wedge\cdots\wedge\nu^{(k-1)}(z)\right],$$
where $\nu^{(j)}$ is the $j-$th derivative of $\nu$ with respect to $z$. Let
$$\Lambda_k(z)=\nu(z)\wedge\cdots\wedge \nu^{(k-1)}(z),$$
and the well-known infinitesimal Pl$\ddot{u}$cker formula (see \cite{gh}) gives,
\begin{align*}
\quad\quad\quad
\frac{\partial^2}{\partial z\partial\overline{z}}\log\|\Lambda_k(z)\|^2=\frac{\|\Lambda_{k-1}(z)\|^2\|\Lambda_{k+1}(z)\|^2}{\|\Lambda_k(z)\|^4}~
\mathrm{for}~k=1,2,\cdots,n,\quad\quad (1.16)
\end{align*}
where we define the norm $\|\cdot\|^2=\langle\cdot,\cdot\rangle$ by the Fubini-Study metric in $\mathbb{CP}(\Lambda^{k}\mathbb{C}^{n+1})$ and put $\|\Lambda_0(z)\|^2=1$. We observe that $(1.16)$ holds only for $\|\Lambda_k(z)\|>0$, i.e., all the unramificated points $z$. Let us set
$\|\Lambda_{n+1}(z)\|=1$ by normalization (analytical extended at the ramificated points) and
\begin{align*}
U_k(z)=-\log\|\Lambda_k(z)\|^2+k(n-k+1)\log2,\quad 1\leq k\leq n.
\end{align*}
Then, from $(1.16)$ we have
\begin{align*}
\Delta U_i+\exp\left(\sum_{j=1}^nK_{ij}U_j\right)-K_0=0~\mathrm{in}~M\setminus S,
\end{align*}
where $K_0$ is the Gaussian curvature, $S$ denotes the set of all the ramificated points of $f$ in $M$ and $\textbf{K}=(K_{ij})_{n\times n}=
\left(\begin{array}{lllll}
~2&-1&~0&\cdots&~0\\
-1&~2&-1&\cdots&~0\\
~0&-1&~2&\cdots&~0\\
~\vdots&~\vdots&~\vdots&~\vdots&~\vdots\\
~0&\cdots&~0&-1&~2
\end{array}\right).
$\\ Near each $p\in S$, we have $U_i=2\gamma_{p,i}\log|z-p|+O(1).$ Thus, $U_i$ satisfies
\begin{align*}
\quad\quad\quad\quad\quad\quad\
\Delta U_i+\exp\left(\sum_{j=1}^nK_{ij}U_j\right)-K_0=4\pi\sum_{p\in S}\gamma_{p,i}\delta_{p},\quad\quad\quad\quad\quad\quad\  (1.17)
\end{align*}
where $\gamma_{p,i}$ stands for the total ramification index at $p$.
\medskip

%We can extend the equation to be defined globally on $M$.
Let $u_i^*=\sum_{j=1}^nK_{ij}U_j$, $\alpha_{p,i}=\sum_{j=1}^nK_{ij}\gamma_{p,j}$. Then it is easy to see that $u_i^*$ satisfies
\begin{equation*}
\Delta u_i^*+\sum_{j=1}^nK_{ij}\left(e^{u_j^*}-K_0\right)=4\pi\sum_{p\in S}\alpha_{p,i}\delta_p,~i=1,2,\cdots,n.
\end{equation*}
When $(M,g)$ is the standard two dimensional sphere with $vol(\mathbb{S}^2)=1$. Then the above equation is
\begin{equation*}
\quad\quad\quad\quad\
\Delta u_i^*+\sum_{j=1}^nK_{ij}\left(e^{u_j^*}-4\pi\right)=4\pi\sum_{p\in S}\alpha_{p,i}\delta_p,~i=1,2,\cdots,n.\quad\quad\quad\quad\   (1.18)
\end{equation*}
Therefore any holomorphic curve from $\mathbb{S}^2$ to $\mathbb{CP}^n$ associates with a solution $\mathbf{u}^*=(u_1^*,\cdots,u_n^*)$ of $(1.18)$. Conversely, given any solution $\mathbf{u}^*=(u_1^*,\cdots,u_n^*)$ of $(1.18)$ in $\mathbb{S}^2$, we can construct
holomorphic curves of $\mathbb{S}^2$ into $\mathbb{CP}^n$ which has the given ramification index $\alpha_{p,i}$ at $p$. One can see
\cite{lwy} for details of the proof. When $n=2$ and $S_1=S_2=S$,  by integrating $(1.17)$, it is easy to see $(1.17)$ can be written as the form of $(1.13)$ with
\begin{equation*}
\quad\quad\quad\quad\quad\quad\quad\quad\quad\quad\quad\quad\quad
\rho_i=4\pi\Big(1+\sum_{j=1}^2K^{ij}N_j\Big),\quad\quad\quad\quad\quad\quad\quad\quad\quad  (1.19)
\end{equation*}
where $(K^{ij})_{2\times 2}=(K_{ij})_{2\times 2}^{-1}$ and $N_i=\sum_{p\in S}\alpha_{p,i}$.
\medskip

System $(1.13)$ also appears in many other problems which arise in geometry and physics. For example,  when  $(1.13)$ is reduced to  the single
equation
\eqref{1.7},  it is related to the Nirenberg problem of finding the prescribing Gaussian curvature if $S_0=\emptyset$, and the existence of a positive constant curvature metric with conic singularities if $S_0\neq\emptyset$. Equation (\ref{1.7}) has been extensively studied during the past decades (see \cite{bt1, bclt, bm, cgy, cl1, cl2, cl3, cl4,  l1,lw0,ly1, nt1,nt2,t,y} and the references therein). Recently,  {it turns out that the} equation (\ref{1.7}) has a deep relation with the classical Lame equation and the Painleve VI equation. We refer the interested readers to \cite{clw,ckl} for the details about the connection. For the general Toda system $(1.13)$, we can also find it in the gauge theory in many physics models. For example, to describe the physics of high critical temperature superconductivity, a model of relative Chern-Simons model was proposed and this model can be reduced to a $n$ times $n$ system with exponential nonlinearity if the gauge potential and the Higgs field are algebraically restricted. Then the Toda system $(1.13)$ is one of the limiting equations if the coupling constant tends to zero. For the detail discussion between the Toda system and its background in Physics, we refer the readers to \cite{d, y} and the references therein. For the   developments of Toda system and general Liouville system, see \cite{m5,d,   jw, jlw, ll,   lwz0, lwz1, lwz2, ly2, lz1, lz2, lz3, m3, m4, m6, nt3,   y1} and references therein.

\medskip

In order to compute the Leray-Schauder degree of the system $(1.14)$, we have to determine the set of parameters $(\rho_1,\rho_2)$  such that the a priori bounds for the solutions of $(1.14)$ might fail. Recently, Lin, Wei and Zhang considered this problem and obtained the following result.

\vspace{0.5cm}

\noindent {\bf{Theorem C}}. (\cite{lwz}) {\em  Let $u=(u_1,u_2)$ be a solution of $(1.14)$ with all $\alpha_{p,i}\in\mathbb{N}$. Suppose   $\rho_1, \rho_2\notin4\pi\mathbb{N}$. Then,
$$\|u_1\|_{L^{\infty}}+\|u_2\|_{L^{\infty}}\leq C$$
for a constant $C$ that only depends on $\rho_i,h_i,\alpha_{p,i}$ and $M.$  }

\vspace{0.5cm}

 To obtain the a priori estimate for all the solutions of $(1.14)$, Lin, Wei and Zhang classified all  the local mass at each blow up point of a sequence of blow up solutions $(u_{1k},u_{2k})$ of $(1.14)$ with $\rho_k=(\rho_{1k},\rho_{2k})$ tends to $\rho=(\rho_1,\rho_2)$. The local mass is defined by
 $$\sigma_i(p)=\lim_{r\rightarrow 0}\lim_{k\rightarrow+\infty}\frac{1}{2\pi}\int_{B_r(p)}\rho_ih_ie^{\tilde{u}_{ik}}dx,~i=1,2,$$
 where
 $$\tilde{u}_{ik}=u_{ik}-\log\Big(\int_Mh_ie^{u_{ik}}\Big).$$
 \begin{remark} We note that $(\sigma_1(p),\sigma_2(p))\neq(0,0)$ if and only if $p$ is a  blow up point. The sufficient part is trivial, but
 the necessary part can follow from the Brezis-Merle Theorem. The argument is standard now. For the reader's convenience,  we sketch it briefly in section 2.
 \end{remark}
 Very recently, Lin, Wei,  and Zhang in \cite{lwz} proved that

 \vspace{0.5cm}

\noindent {\bf{Theorem D}}. (\cite{lwz}) {\em  Suppose $p$ is a blow up point of a sequence of blow up solutions of $(1.14)$ with all $\alpha_{p,i}\in\mathbb{N}$, $i=1,2$.    Then   $\sigma_1(p), \sigma_2(p)\in 2\mathbb{N}\cup\{0\}$.}

 \vspace{0.5cm}

Theorem D  implies that if $S_1\cup S_2=\emptyset$, then
 \begin{itemize}\item[(i)] if  $\textbf{K}=\textbf{A}_2$, then
 \begin{align*}(\sigma_1(p),\sigma_2(p))\in\Big\{(2,0),(0,2),(2,4),(4,2),(4,4)\Big\}.\end{align*}
 \item[(ii)] if $\textbf{K}=\textbf{B}_2$, then
 \begin{align*}
 (\sigma_{1},\sigma_{2})\in\Big\{(2,0), (0,2),   (4,2), (2,6),  (4,8), (6,6), (6,8)\Big\}.
 \end{align*}
  \item[(iii)] if $\textbf{K}=\textbf{G}_2$,     then
 \begin{align*}
 (\sigma_{1},\sigma_{2})\in\Big\{(2,0), (0,2),  (2,8),  (4,2),  (12,18), (12,20), \\ (4,12), (8,8), (8,18), (10,12), (10,20) \Big\}.
 \end{align*}
 \end{itemize}
This generalizes an earlier result by Lin and Zhang \cite{lz4}. We notice that for $(1.14)$ with singular sources, the number of the possibility of the local mass relies heavily on the coefficients $\alpha_{p,i}$ of the singular source, as the coefficient becomes larger, the number of possibility gets bigger. This would increase the difficulty in analyzing the bubbling solution for $(1.14)$.
\medskip

By Theorem C, we can define the Leray-Schauder degree $d_{\rho_1,\rho_2}^{\mathbf{K}}$ for $(1.14)$ when $\rho_1\in(4i\pi,4(i+1)\pi)$ and $\rho_2\in(4j\pi,4(j+1)\pi),i,j\in\mathbb{N}\cup\{0\}$ and $\mathbf{K}=\mathbf{A}_2,\mathbf{B}_2$ or $\mathbf{G}_2$. Again the degree is a homotopic invariant and is a constant when $\rho_1\in(4i\pi,4(i+1)\pi)$ and $\rho_2\in(4j\pi,4(j+1)\pi),i,j\in\mathbb{N}\cup\{0\}$. We denote it by $d_{i,j}^{\mathbf{K}}$. Then we introduce the generating function $g_i^{(2)}(x,\mathbf{K}):$
\begin{align*}\quad\quad\quad\quad\quad\quad\quad\quad\quad\quad\quad\quad\   g_i^{(2)}(x,\mathbf{K})=\sum_{j=0}^{\infty}d_{i,j}^{\mathbf{K}}x^j.\quad\quad\quad\quad\quad\quad\quad\quad\quad\quad\quad (1.20)\end{align*}
Obviously, $g_0^{(2)}(x)=g^{(1)}(x)$, where $g^{(1)}(x)$ is given by \eqref{1.9} with   $S_0=S_2$. So far, the first three authors with Wei \cite{llwy} obtained $g_1^{(2)}(x)$ when $S_1\cup S_2=\emptyset$ in the following theorem.
\medskip

\noindent {\bf{Theorem E}}. (\cite[Theorem 1.6]{llwy}) {\em Let $g^{(2)}_1(x,\mathbf{K})$ be the generating function defined above. Suppose $S_1\cup S_2=\emptyset.$ Then the generating function $g^{(2)}_1(x,\mathbf{K})$ is determined by,
\begin{align*}
g^{(2)}_1(x,\mathbf{K})=\sum_{j=0}^\infty d_{1,j}^{\textbf{K}}x^j=(1-x)^{\chi(M)-1}\left(1-\chi(M)(1+x+\cdots+x^{-K_{21}})\right),
\end{align*}
}

\begin{remark} We can also define the generating function
$$\tilde{g}_i^{(2)}(x,\mathbf{K})=\sum_{j=0}^{\infty}d_{j,i}^{\mathbf{K}}x^j.$$
It is easy to see $\tilde{g}_0^{(2)}(x,\mathbf{K})=g^{(1)}(x)$, where $g^{(1)}(x)$ is given by \eqref{1.9} with  $S_0=S_1$.  As for $g_1^{(2)}(x,\mathbf{K})$, we can also derive $\tilde{g}_i^{(2)}(x,\mathbf{K})$ when $S_1\cup S_2=\emptyset$,
\begin{align*}
\tilde{g}_i^{(2)}(x,\mathbf{K})=\sum_{j=0}^\infty d_{j,1}^{\textbf{K}}x^j=(1-x)^{\chi(M)-1}\left(1-\chi(M)(1+x+\cdots+x^{-K_{12}})\right).
\end{align*}
See \cite{llwy} for details.
\end{remark}

In the present paper, we want to extend Theorem E for the system $(1.14)$ when $S_1\cup S_2\neq\emptyset$. Following   \cite{llwy}, we have to compute the gap between $d_{0,j}^{\mathbf{K}}$ and $d_{1,j}^{\mathbf{K}}$.  Our strategy is to reduce the computation of the gap to a single equation. More precisely,  we consider all the bubbling solutions of $(1.14)$ when $\rho_2\notin4\pi\mathbb{N}$ is fixed and $\rho_{1k}\rightarrow4\pi$ from below or above of $4\pi$. Then we can show that $u_{1k}$ blows up at $Q\notin S_1$ and $u_{2k}$ converges to $w+4\pi K_{21} G(x,Q)$ in $C^{2,\alpha}_{loc}(M\setminus\{Q\})$, where $(w,Q)$ satisfies the  shadow system \eqref{1.1}.
In conclusion, we have the following theorem,
\begin{theorem}
\label{th1.1}
Suppose  $h_i$ satisfies $(1.15)$ with $\alpha_{p,i}\in\mathbb{N},~i=1,2.$   Let   $(u_{1k},u_{2k})$ be a sequence of solutions of $(1.14)$ with $(\rho_{1k},\rho_{2k})\rightarrow(4\pi,\rho_2)$ satisfying $\rho_2\notin4\pi\mathbb{N}$ and $\max_{M}(u_{1k},u_{2k})\to+\infty$. Then, we have
\[\rho_{1k}\frac{h_1e^{u_{1k}}}{\int_Mh_1e^{u_{1k}}}\rightarrow 4\pi\delta_Q,\   Q\in M\setminus S_1,\  \textrm{and}\] \[u_{2k}\rightarrow w+4\pi K_{21} G(x,Q)\ \textrm{ in}\ C^{2,\alpha}_{loc}(M\setminus\{Q\}),\]
where $(w,Q)$ is a solution of (\ref{1.1}).
\end{theorem}

  Once we get the degree $d_j^S$ for the shadow system (\ref{1.1}) by Theorem \ref{th1.2},  as a consequence, we are able to obtain the degree gap between $d_{0,j}^{\mathbf{K}}$ and $d_{1,j}^{\mathbf{K}}$ by the following result.

\medskip

\noindent {\bf{Theorem F}}. (\cite[Theorem 1.4]{llwy}) {\em We have
 \begin{equation*} d_{1,j}^{\textbf{K}}-d_{0,j}^{\textbf{K}}=-d^S_j,\end{equation*}
}

\medskip

Now we can obtain the generating function $g_1^{(2)}(x,\mathbf{K})$  for $(1.14)$ as follows.

\begin{theorem}
\label{th1.4}
Suppose that  $h_i$ satisfies $(1.15)$ with $\alpha_{p,i}\in\{1,2\},~i=1,2.$
 Then the generating function $g_1^{(2)}(x,\mathbf{K})$  given by $(1.20)$ can be represented by
\begin{align*}
g_1^{(2)}(x,\mathbf{K})=\sum_{k=0}^{\infty}d_{1,k}^{\mathbf{K}}x^k=(1-x)^{\chi(M)-1}\prod_{p\in S_2}(1+x+\cdots+x^{\alpha_{p,2}})-g_s(x),
\end{align*}
where $g_s(x)$ is given in (\ref{1.15}).
\end{theorem}

As a consequence of Theorem \ref{th1.4}, we have the following corollaries.
\begin{corollary}
\label{cr1.2}
Suppose  that the assumption in Theorem \ref{th1.4} holds. If $\chi(M)\leq0$, then system $(1.14)$ always has a solution when $\rho_1\in(0,4\pi)\cup(4\pi,8\pi)$, $\rho_2\notin4\pi\mathbb{N}$.
\end{corollary}

\noindent For equation $(1.18)$ with $n=2$ and $\textbf{K}=\textbf{A}_2$, we recall that
$$N_1=\sum_{p\in S_1}\alpha_{p,1}~\mathrm{and}~N_2=\sum_{p\in S_2}\alpha_{p,2},$$
where $S_1=S_2=S.$ By $(1.19)$, if $N_1\not\equiv N_2\mod3$, then $\rho_i\notin 4\pi\mathbb{N}$. Suppose that $\rho_1<8\pi$. Then we have
\begin{corollary}
\label{cr1.3}
Suppose $M=\mathbb{S}^2,$ $\textbf{K}=\textbf{A}_2,$  $S_1=\emptyset$,   $|S_2|=1,2$, and  $\alpha_{p,2}=1$ for any $p\in S_2$, then  equation $(1.18)$ has a solution.
\end{corollary}

%\begin{align*}
%f_1(x)=-(1-x)(x+\cdots+x^{-K_{21}}).
%\end{align*}
%
%\begin{align*}
%f_2(x)=-(1-x)(1+3x+\cdots+3x^{-K_{21}}+2x^{1-K_{21}}).
%\end{align*}
%
%\begin{align*}
%f_3(x)=-(1-x)(1+4x+\cdots+4x^{1-K_{21}}+2x^{2-K_{21}}).
%\end{align*}

\begin{remark} We can also derive the expression of $\tilde{g}_1^{(2)}(x,\mathbf{K})$ in the following
\begin{align*}
\tilde{g}_1^{(2)}(x,\mathbf{K})=~&(1-x)^{\chi(M)-1}\Big[\prod_{p\in S_1}(1+x+\cdots+x^{\alpha_{p,1}})\\
&-\left(\chi(M)-|S_1\cup S_2|\right)(1+\cdots+x^{-K_{12}})\prod_{p\in S_1}(1+\cdots+x^{\alpha_{p,1}})\\
&-\sum_{p\in S_1\setminus S_2}(1+x+\cdots+x^{\alpha_{p,1}-K_{12}})\prod_{q\in S_1\setminus\{p\}}(1+x+\cdots+x^{\alpha_{q,1}})\Big].
\end{align*}
\end{remark}

This paper is organized as follows. In section 2, we derive the shadow system (\ref{1.1}) from the bubbling solutions of $(1.14)$ as $\rho_1$ tends to $4\pi$. In section 3, we prove the compactness of the solutions of (\ref{1.1}) in $\mathring{H}^1(M)\times [M\setminus S_1]$. In section 4, we obtain the result for the local mass when  there are collapsing singularities. In section 5, we study the deformation $(1.10)_t$, prove the compactness of the solutions, and derive the topological degree of (\ref{1.1}). In section 6, we state some applications of the degree formula of (\ref{1.1}).

\vspace{1cm}

\section{Shadow system with singular sources}\label{section2}
Let   $(u_{1k},u_{2k})\in\mathring{H}^1(M)\times\mathring{H}^1(M)$ be a solution  of  $(1.14)$  with $(\rho_{1k},\rho_{2k})$ such that  $\max_{M}(u_{1k},u_{2k})\rightarrow+\infty$ as $k\to+\infty$.
We  set
\begin{equation}
\label{2.1}\tilde{u}_{ik}:=u_{ik}-\ln\int_M h_ie^{u_{ik}}dv_g,~i=1,2.\end{equation} Then $(\tilde{u}_{1k},\tilde{u}_{2k})$
satisfy
\begin{equation}
\label{2.2}
\left\{\begin{array}{ll}
\Delta \tilde{u}_{1k}+2\rho_{1k}(h_1e^{\tilde{u}_{1k}}-1)+K_{12}\rho_{2k}(h_2e^{\tilde{u}_{2k}}-1)=0,\\
\Delta \tilde{u}_{2k}+2\rho_{2k}(h_2e^{\tilde{u}_{2k}}-1)+K_{21}\rho_{1k}(h_1e^{\tilde{u}_{1k}}-1)=0.
\end{array}\right.
\end{equation}  From   \eqref{2.1}, we see that
\begin{equation}\label{2.3}\int_M h_1e^{\tilde{u}_{1k}}dv_g=\int_M h_2e^{\tilde{u}_{2k}}dv_g=1. \end{equation}
We define the blow up set for $\tilde{u}_{ik}$
\begin{align}
\label{2.4}
\mathfrak{S}_{i}:=\{p\in M\mid \exists \{x_k\},~x_k\rightarrow p,~\tilde{u}_{ik}(x_k)\rightarrow+\infty\} \ \textrm{for}\ i=1,2,
\end{align}
and \[\mathfrak{S}:=\mathfrak{S}_1\cup\mathfrak{S}_2.\]
 For any $p\in M,$ we define the local mass by
\begin{equation}
\label{2.7}
\sigma_{i}(p)=\lim_{\delta\rightarrow0}\lim_{k\rightarrow+\infty}\frac{1}{2\pi}\int_{B_{\delta}(p)}\rho_{ik}h_ie^{\tilde{u}_{ik}}dv_g, \ i=1,2.
\end{equation}
We also denote $\gamma_i(p)$, $i=1,2$ such that \begin{equation}
\label{2.11}
\gamma_i(p)=\left\{\begin{array}{ll}
 \alpha_{p,i}\ \ \textrm{if}\ \ p\in S_i,\\
0\ \ \ \ \  \textrm{if}\ \ p\notin S_i.
\end{array}\right.
\end{equation}
It was proved in \cite{lwz} that $(\sigma_1(p),\sigma_2(p))$ satisfies the following Pohozaev identity:
\begin{equation}\begin{aligned}\label{poho}&K_{21}\sigma_1^2(p)+K_{12}K_{21}\sigma_1(p)\sigma_2(p)+K_{12}\sigma_2^2(p)\\=~&2K_{21}(1+\gamma_1(p))\sigma_1(p)
+2K_{12}(1+\gamma_2(p))\sigma_2(p),
\end{aligned}\end{equation} {where} $\textbf{K}=\textbf{A}_2$, $\textbf{B}_2$, $\textbf{G}_2$.
\medskip

For $\sigma_i(p)$, we have the following lemma.
\begin{lemma}\label{le2.1}$p\notin\mathfrak{S}$ if and only if $\sigma_1(p)=\sigma_2(p)=0$.
\end{lemma}
\begin{proof} The proof   is well-known now, and it  follows from the  Brezis-Merle's result \cite{bm}. We give a sketch here for convenience of readers.

 We  note that if $p\notin\mathfrak{S}$, then there is a neighborhood $U$ of $p$ such that $\tilde{u}_{1k}$ and $\tilde{u}_{2k}$ are uniformly bounded from above by a constant, independent of $k$. So we can get $\sigma_1(p)=\sigma_2(p)=0$ easily.

If $\sigma_1(p)=\sigma_2(p)=0$, then
 we can choose small $r_0>0$ such that
\begin{equation}
\label{2.6}
 \int_{B_{r_0}(p)}\rho_ih_i|e^{\tilde{u}_{ik}}-1|dx<\frac{\pi}{6},\ i=1,2.
\end{equation}
For $i=1,2$, let $\eta_{ik}$ be a harmonic function in $B_{r_0}(p)$ with  $\eta_{ik}=\tilde{u}_{ik}$ on $\partial B_{r_0}(p)$.
In view of \cite[Theorem 1]{bm} and \eqref{2.6}, we can find   some  constants $\delta, C_\delta>0$, independent of $k$, such that
\begin{equation}
\label{02.7}
\int_{B_{r_0}(p)}\exp((1+\delta)|\tilde{u}_{ik}-\eta_{ik}|)dx\leq C_\delta, \ i=1,2.
\end{equation}From the mean value theorem for harmonic function and  \eqref{2.6}-\eqref{02.7}, we can get a constant $c>0$, independent of $k$, such that
\begin{equation}
\label{2.8}
\eta_{ik}^+\leq c\ \textrm{in}\ B_{\frac{r_0}{2}}(p),\ i=1,2.
\end{equation}
By using the standard elliptic estimate and \eqref{02.7}-\eqref{2.8}, we can get  that $\tilde{u}_{1k}$ and $\tilde{u}_{2k}$ are uniformly bounded from above in
$B_{\frac{r_0}{2}}(p)$. Hence,   $p\notin\mathfrak{S}$ .
\end{proof}
From the proof of Lemma \ref{le2.1}, we see that if $p\in\mathfrak{S}$, then \eqref{2.6} does not hold. Then using  the fact $\int_Mh_ie^{\tilde{u}_{ik}}dv_g=1$,  we have
\[|\mathfrak{S}|<+\infty.\] Let $r_0>0$ be small enough such that $B_{4r_0}(p)\cap B_{4r_0}(q)=\emptyset$ for $p\neq q\in\mathfrak{S}$, and we have the following result.
\begin{lemma}\label{le2.2}For $1\le i\le 2$, \[p\notin\mathfrak{S}_i\ \textrm{ if and only if}\  \sigma_i(p)=0.\]
\end{lemma}
\begin{proof} If $p\notin\mathfrak{S}_i$, then there is a neighborhood $U$ of $p$ such that $\tilde{u}_{ik}$ is uniformly bounded from above by a constant, independent of $k$. So we  get $\sigma_i(p)=0$.

 Now we suppose that $\sigma_i(p)=0$. There is a constant $c>0$, independent of $k$, such that
\begin{equation}\sup_{\partial B_{r_0}(p)}\tilde{u}_{ik}\le c.\label{2.14}\end{equation}
Let $1\le j\neq i\le 2$ and  $\phi_k$ satisfy $\Delta \phi_k+K_{ij}\rho_{jk}h_je^{\tilde{u}_{jk}}=0$ in $B_{r_0}(p)$ and $\phi_k=\tilde{u}_{ik}$ on $\partial B_{r_0}(p)$.
The maximum principle and \eqref{2.14} imply that $\phi_k$ is uniformly bounded from above in $B_{r_0}(p)$.
We note that $\hat{u}_{ik}=\tilde{u}_{ik}-\phi_k$ satisfies
\begin{equation*}
\left\{\begin{array}{ll}
\Delta \hat{u}_{ik}+2\rho_{ik}(h_ie^{\phi_k}e^{\hat{u}_{ik}}-1)-K_{ij}\rho_{jk}=0\  \textrm{in}\  B_{r_0}(p),\\
\hat{u}_{ik}=0 \  \textrm{on}\   \partial B_{r_0}(p).
\end{array}\right.
\end{equation*}
By applying \cite[Theorem 1]{bm} to $\hat{u}_{ik}=\tilde{u}_{ik}-\phi_k$ as in Lemma \ref{le2.1}, we get that $\tilde{u}_{ik}$ is uniformly bounded from above in $B_{r_0}(p)$. Therefore, $p\notin \mathfrak{S}_i$.
\end{proof}
 \begin{remark}\label{re1}
 In view of  the Green representation formula and the elliptic estimates, it is easy to see that
for any compact set $K\subset\subset M\setminus\mathfrak{S}$, there is a constant $C_K>0$,  independent of $k$,
 satisfying
\begin{align}
\label{2.9}
\|u_{ik}\|_{L^\infty(K)}\leq C_K\ \ \textrm{for all} \ \ k\ge1,\ i=1,2,
\end{align}
and
\begin{align}
\label{2.10}
|\tilde{u}_{ik}(x)-\tilde{u}_{ik}(y)|\leq C_K\ \ \textrm{for all} \ \ k\ge1,\ \ \mbox{any}\ x, y\in K, \ i=1,2.
\end{align}
\end{remark}
Since we assume that   $\max_{M}(u_{1k},u_{2k})\rightarrow+\infty$,  \eqref{2.9} implies $\max_{M}(\tilde{u}_{1k},\tilde{u}_{2k})\rightarrow+\infty$, that is, \[\mathfrak{S}=\mathfrak{S}_1\cup\mathfrak{S}_2\neq\emptyset.\]
For $p\in\mathfrak{S}$, we have the following result.
\begin{lemma}\label{atleast}
For $p\in\mathfrak{S}$, we have \[\textrm{either}\ 2\sigma_1(p)-2\gamma_1(p)+K_{12}\sigma_2(p)\ge 2\ \textrm{or}\ 2\sigma_2(p)-2\gamma_2(p)+K_{21}\sigma_1(p)\ge 2.\]
\end{lemma}
\begin{proof}From \eqref{poho}, we get that
\begin{equation}\begin{aligned}\label{frompoho1}&K_{21}\sigma_1^2(p)+K_{12}K_{21}\sigma_1(p)\sigma_2(p)+K_{12}\sigma_2^2(p)\\
=~&K_{21}\sigma_1(p)\Big(2\sigma_1(p)-2-2\gamma_1(p)+K_{12}\sigma_2(p)\Big)\\
&+K_{12}\sigma_2(p)\Big(2\sigma_2(p)
-2-2\gamma_2(p)+K_{21}\sigma_1(p)\Big).
\end{aligned}\end{equation}
Since $4-K_{21}K_{12}>0$ and $K_{12}, K_{21}<0$, we see
\begin{equation}\begin{aligned}\label{disc}0&>K_{21}\Big(\sigma_1+\frac{K_{12}\sigma_2(p)}{2}\Big)^2+K_{12}\Big(\frac{4-K_{21}K_{12}}{4}\Big)\sigma_2^2(p)
\\&=K_{21}\Big(\sigma_1^2(p)+K_{12}\sigma_1(p)\sigma_2(p)+\frac{K_{12}^2\sigma_2^2(p)}{4}\Big)+K_{12}\Big(\frac{4-K_{21}K_{12}}{4}\Big)\sigma_2^2(p)
\\&=K_{21}\sigma_1^2(p)+K_{12}K_{21}\sigma_1(p)\sigma_2(p)+K_{12}\sigma_2^2(p).
\end{aligned}\end{equation}
  If $2\sigma_1(p)-2\gamma_1(p)+K_{12}\sigma_2(p)< 2$  and $2\sigma_2(p)-2\gamma_2(p)+K_{21}\sigma_1(p)< 2$, then we get a contradiction from \eqref{frompoho1} and \eqref{disc}. So we complete the proof of Lemma \ref{atleast}.
\end{proof}
In \cite[Lemma 2.1]{llwy}, it was proved that if $S_1\cup S_2=\emptyset$, then     a  weak  concentration phenomena holds (i.e.
 if a sequence of solutions $(u_{1k},u_{2k})$ of $(1.14)$ blows up, then one of  $\frac{h_ie^{u_{ik}}}{\int_Mh_ie^{u_{ik}}dv_g}$, $i=1,2$, tends to a sum of Dirac measures).
 Now we are going to extend this result to the general cases.
\begin{lemma}\label{weakconcentration}
 If $p\in\mathfrak{S}$ and  $2\sigma_1(p)-2\gamma_1(p)+K_{12}\sigma_2(p)\ge 2$,  then
\begin{align}\label{2.13}
  \tilde{u}_{1k}\to-\infty\ \mbox{uniformly in any compact subset of} \  B_{r_0}(p)\setminus\{p\}\ \textrm{as} \ k\to+\infty.
\end{align}
\end{lemma}
\begin{proof}
To prove  \eqref{2.13}, we argue by
contradiction. Then for any fixed $r_1\in(0,r_0)$, we see that $\sup_{r_1\le |x-p|\le r_0}\tilde{u}_{1k}$ is uniformly bounded from below by some constant depending on $r_1$, not on $k$.  From \eqref{2.10},
$\inf_{r_1\le |x-p|\le r_0}\tilde{u}_{1k}$ is also uniformly bounded from below.  Together,   we get  that
 $\sup_{r_1\le |x-p|\le r_0}|\tilde{u}_{1k}|$ is  uniformly bounded. Since $r_1>0$ is arbitrary,  the standard elliptic estimates implies that $\tilde{u}_{1k}\to\xi_1$ in $C^2_{\textrm{loc}}(B_{r_0}(p)\setminus\{p\})$ as $k\to+\infty$
 for some function $\xi_1$. Since $\int_Mh_1e^{\tilde{u}_{1k}}dv_g=1$, we also  see that $h_1e^{\xi_1}\in L^1(B_{r_0}(0))$.
 We need to consider the following two cases according to the asymptotic behavior of $\tilde{u}_{2k}$.

 \medskip

\noindent Case 1.  $\sup_{r\le |x-p|\le r_0}\tilde{u}_{2k}\to-\infty$ as $k\to+\infty$ for any fixed $r\in(0,r_0]$.

  \medskip
 \noindent
By using $\tilde{u}_{1k}\to\xi_1$ in $C^2_{\textrm{loc}}(B_{r_0}(p)\setminus\{p\})$ as $k\to+\infty$,  we see that  $\xi_1$ satisfies
\[
\Delta \xi_1+2\rho_1(h_1e^{\xi_1}-1)=  -2\pi(2\sigma_1(p)+K_{12}\sigma_2(p))\delta_p\ \mbox{in}\  B_{r_0}(p).
\]
 By using Green representation formula, we get that for $x\in B_{r_0}(p)$,
 \begin{equation}\begin{aligned}\label{2.16}
 \xi_1(x)=&-(2\sigma_1(p)+K_{12}\sigma_2(p))\ln|x-p|\\&+\frac{1}{2\pi}\int_{B_{r_0(p)}}\ln\frac{1}{|x-y|} 2\rho_1(h_1e^{\xi_1(y)}-1)dy
 \\&+\int_{\partial B_{r_0(p)}}\Big[\frac{\xi_1(y)}{2\pi}\frac{\partial \ln|x-y|}{\partial\nu}-\frac{1}{2\pi}\ln|x-y|\frac{\partial \xi_1}{\partial\nu}\Big]dS.
\end{aligned}
\end{equation}  Let $\eta(x)=\int_{\partial B_{r_0(p)}}\Big[\frac{\xi_1(y)}{2\pi}\frac{\partial \ln|x-y|}{\partial\nu}-\frac{1}{2\pi}\ln|x-y|\frac{\partial \xi_1}{\partial\nu}\Big]dS$. Then we see that $\eta\in C^1(B_{\frac{r_0}{2}}(p))$.   Since $h_1e^{\xi_1}\in L^1(B_{r_0}(p))$,   if $|x-p|\le\frac{r_0}{2}$,  then
 \begin{equation*}\begin{aligned}
 \xi_1(x)\ge&-(2\sigma_1(p)+K_{12}\sigma_2(p))\ln|x-p|+\frac{(-\ln{2r_0})}{2\pi} \|2\rho_1h_1e^{\xi_1}\|_{B_{r_0(p)}}
 \\&+\frac{  (-\rho_1)\|\ln {|y|}\|_{L^1(B_{2r_0}(0)) }}{\pi}+\eta(x)\ge-(2\sigma_1(p)+K_{12}\sigma_2(p))\ln|x-p|+ c,
\end{aligned}
\end{equation*}where $c$ is a constant, independent of $x\in B_{\frac{r_0}{2}}(p)$.\\
In view of  $h_1e^{\xi_1}\in L^1(B_{r_0}(p))$, $h_1(x)=h_1^*(x)e^{-\sum_{q\in S_1}4\pi\gamma_1(q)G(x,q)}$, and $h_1^*>0$, we get that \begin{equation}\label{2.17}2+2\gamma_1(p)-2\sigma_1(p)-K_{12}\sigma_2(p)>0, \end{equation}which contradicts   the assumption in Lemma \ref{weakconcentration}.

 \medskip

\noindent Case 2. $\sup_{r\le |x-p|\le r_0}\tilde{u}_{2k}$ is uniformly bounded from below for each  $r\in(0,r_0]$.

 \medskip\noindent

 From the similar arguments in Case 1, there is a function $\xi_2$ satisfying  $\tilde{u}_{2k}\to\xi_2$ in $C^2_{\textrm{loc}}(B_{r_0}(p)\setminus\{p\})$ as $k\to+\infty$ and
 $h_2e^{\xi_2}\in L^1(B_{r_0}(0))$. On $B_{r_0}(p)$, we have
\begin{equation*}
\left\{\begin{array}{ll}
\Delta \xi_1+2\rho_1(h_1e^{\xi_1}-1)+K_{12}\rho_2(h_2e^{\xi_2}-1) =-2\pi(2\sigma_1(p)+K_{12}\sigma_2(p))\delta_p,\\
\Delta \xi_2+2\rho_2(h_2e^{\xi_2}-1)+K_{21} \rho_1(h_1e^{\xi_1}-1) =-2\pi(2\sigma_2(p)+K_{21}\sigma_1(p))\delta_p.\\
\end{array}\right.
\end{equation*}
Now we consider the following two cases (i)-(ii) according to the value of $2\sigma_2(p)-2\gamma_2(p)+K_{21}\sigma_1(p)$:

\medskip

\noindent (i)  If $2\sigma_2(p)-2\gamma_2(p)+K_{21}\sigma_1(p)<2$.
\medskip

\noindent At first, we claim that \begin{equation}h_2e^{\xi_2}\in L^{1+\delta_0}(B_{\tau_0}(p))\ \ \textrm{for some} \ \delta_0>0,\ \tau_0\in(0,\frac{r_0}{2}).\label{2.18}\end{equation}
Let \[\zeta_2={\xi}_2+(2\sigma_2(p)+K_{21}\sigma_1(p))\ln|x-p|.\] Then $\zeta_2$ satisfies
\begin{equation*}
\Delta \zeta_2+2\rho_2(h_2e^{{\xi}_2}-1)+K_{21} \rho_1(h_1e^{\xi_1}-1)=0.
\end{equation*}
We note that for any small $r>0$, there is a constant $c_r>0$ such that
\[\sup_{\partial B_{r}(p)}\zeta_2\le c_r.\]
Since $h_ie^{\xi_i}\in L^1(B_{r_0}(p))$ for $i=1,2$, by using \cite[Theorem 1]{bm} as in Lemma \ref{le2.1}, we see that for any $\delta>0$, there are constants $C_\delta>0$  and $\tau_{\delta}\in(0,\frac{r_0}{2})$ satisfying
\begin{equation}
\label{2.19}
\int_{B_{\tau_{\delta}}(p)}\exp((1+\delta)|\zeta_2|)dv_g\leq C_\delta.
\end{equation} We recall $h_2(x)=h_2^*(x)e^{-\sum_{q\in S_2}4\pi\gamma_2(q)G(x,q)}$ where $h_2^*>0$.
By using \eqref{2.19}  and $2\sigma_2(p)-2\gamma_2(p)+K_{21}\sigma_1(p)<2$, we see that there are  constants $\delta_0>0$, $\tau_0>0$ and a positive function $\bar{h}_2$  such that   \begin{equation*}h_2(x)e^{\xi_2(x)}=\bar{h}_2(x)|x-p|^{-2\sigma_2(p)+2\gamma_2(p)-K_{21}\sigma_1(p)}e^{\zeta_2}\in L^{1+\delta_0}(B_{\tau_0}(p)),\end{equation*} which implies the claim \eqref{2.18}.

 By   Green representation formula as in \eqref{2.16}  and $h_2e^{\xi_2}\in L^{1+\delta_0}(B_{\tau_0}(p))$,  we get a constant $c$, independent of $x\in B_{\tau_0}(p)$, satisfying \begin{equation*}\begin{aligned}
 \xi_1(x)\ge - (2\sigma_1(p)+K_{12}\sigma_2(p))\ln|x-p|+ c.
\end{aligned}
\end{equation*}
In view of  $h_1e^{\xi_1}\in L^1(B_{r_0}(p))$, $h_1(x)=h_1^*(x)e^{-\sum_{q\in S_1}4\pi\gamma_1(q)G(x,q)}$, and $h_1^*>0$, we get that $2+2\gamma_1(p)
-  2\sigma_1(p)-K_{12}\sigma_2(p)>0,$ which contradicts  the assumption in Lemma \ref{weakconcentration}.

\medskip

\noindent (ii) If $2\sigma_2(p)-2\gamma_2(p)+K_{21}\sigma_1(p)\ge 2$.
\medskip

\noindent On $B_{r_0}(p)$, we  have
\begin{equation*}\begin{aligned}
&\Delta \Big(2\xi_2-K_{21}\xi_1\Big)+\rho_2(4-K_{12}K_{21})(h_2e^{\xi_2}-1)\\=&-2\pi\Big[2(2\sigma_2(p)+K_{21}\sigma_1(p))-K_{21}(2\sigma_1(p)+K_{12}\sigma_2(p))\Big]\delta_p.
\end{aligned}\end{equation*}
 By   Green representation formula as in \eqref{2.16},  $h_2e^{\xi_2}\in L^1(B_{\tau}(p))$ and $4-K_{12}K_{21}>0$,
   we get a constant $c$, independent of $x\in B_{\tau}(p)$, satisfying
    \begin{equation*}\begin{aligned}
&2 \xi_2(x) -K_{21} \xi_1(x)  \\&\ge   \Big[-2(2\sigma_2(p) +K_{21}\sigma_1(p))+K_{21}(2\sigma_1(p) +K_{12}\sigma_2(p))\Big] \ln|x-p|+ c.
\end{aligned}
\end{equation*}Then we see that \begin{equation*}\begin{aligned}
&(2-K_{21}) \max\Big\{\xi_1(x)+2\gamma_1(p)\ln|x-p|,\ \xi_2(x)+2\gamma_2(p)\ln|x-p|\Big\}\\&\ge 2\Big\{\xi_2(x)+2\gamma_2(p)\ln|x-p|\Big\}-K_{21}\Big\{\xi_1(x)+2\gamma_1(p)\ln|x-p|\Big\} \\&\ge   \Big\{-2(2\sigma_2(p)-2\gamma_2(p)+K_{21}\sigma_1(p))\\&\quad+K_{21}(2\sigma_1(p)-2\gamma_1(p)+K_{12}\sigma_2(p))\Big\} \ln|x-p|+ c.
\end{aligned}
\end{equation*}
Then there is a constant $C>0$, independent of $x\in B_{\tau}(p)$, satisfying
\begin{equation*}\begin{aligned}
  &|x-p|^{2\gamma_1(p)}e^{\xi_1(x)}+|x-p|^{2\gamma_2(p)}e^{\xi_2(x)}  \\&\ge   e^{\max\{\xi_1(x)+2\gamma_1(p)\ln|x-p|,\xi_2(x)+2\gamma_2(p)\ln|x-p|\}}\\&  \ge C|x-p|^{\frac{-2(2\sigma_2(p)-2\gamma_2(p)+K_{21}\sigma_1(p))+K_{21}(2\sigma_1(p)-2\gamma_1(p)+K_{12}\sigma_2(p))}{2-K_{21}}},
\end{aligned}
\end{equation*}
In view of  $h_ie^{\xi_i}\in L^1(B_{r_0}(p))$, $h_i(x)=h_i^*(x)e^{-\sum_{q\in S_i}4\pi\gamma_i(q)G(x,q)}$, and $h_i^*>0$, we get that \begin{equation*}2-\frac{2(2\sigma_2(p)-2\gamma_2(p)+K_{21}\sigma_1(p))-K_{21}(2\sigma_1(p)-2\gamma_1(p)+K_{12}\sigma_2(p))}{2-K_{21}}>0.\end{equation*}
Since we assume that  $2\sigma_i(p)-2\gamma_i(p)+K_{ij}\sigma_j(p)\ge 2$ where $1\le i\neq j\le 2$,  we see that
\begin{equation*}\begin{aligned}0&=2-\frac{4-2K_{21}}{2-K_{21}}\\&\ge 2-\frac{2(2\sigma_2(p)-2\gamma_2(p)+K_{21}\sigma_1(p))-K_{21}(2\sigma_1(p)-2\gamma_1(p)+K_{12}\sigma_2(p))}{2-K_{21}},\end{aligned}\end{equation*}
which implies a contradiction.

\medskip

At this point, we complete the proof of Lemma \ref{weakconcentration}.\end{proof}
\begin{remark}By using the proof of Lemma \ref{weakconcentration}, we can also show that if $p\in\mathfrak{S}$ and $2\sigma_2(p)-2\gamma_2(p)+K_{21}\sigma_1(p)\ge 2$,  then
$  \tilde{u}_{2k}\to-\infty$ {uniformly in any compact subset of} $B_{r_0}(p)\setminus\{p\}$ {as}  $k\to+\infty.$
\end{remark}

Now we are going to derive the shadow system for the bubbling solutions of $(1.14)$ as $\rho_{1k}\to 4\pi$.

\medskip
 \noindent
\textbf{Proof of Theorem \ref{th1.1}.}  We recall the following assumption:   $\max_{M}(u_{1k},u_{2k})\rightarrow+\infty$,   $\rho_{1k}\to4\pi$,  and
$\rho_{2k}\to\rho_2\notin4\pi\mathbb{N}$  for $1\le i\neq j\le 2$ as $k\to+\infty$.\\
Suppose that $\mathfrak{S}_1=\emptyset$. Then $\mathfrak{S}_2\neq\emptyset$, and   Lemma \ref{atleast}-\ref{weakconcentration} imply that $\tilde{u}_{2k}$ is concentrate.      Theorem D and Lemma \ref{le2.2} imply  that $\sigma_2(p)\in 2\mathbb{N}$ for $p\in\mathfrak{S}_2$. Since $\int_Mh_2e^{\tilde{u}_{2k}}=1$, we get that  $\rho_{2k}\to\rho_2\in4\pi\mathbb{N}$, which contradicts to the assumption $\rho_2\notin4\pi\mathbb{N}$.
So we get that $\mathfrak{S}_1\neq\emptyset$. In view of $\rho_{1k}\to4\pi$, we get that $|\mathfrak{S}_1|=1$ and there is a point $Q\in M$ such that   \[\mathfrak{S}_1=\{Q\},\ \textrm{ and}\  \sigma_1(Q)=2.\]
 If  $\sigma_{2}(Q)>0$, then  Theorem D implies   that $\sigma_{2}(Q)\in 2\mathbb{N}$. If $\sigma_{1}(Q)=\sigma_{2}(Q)=2$,  then the Pohozaev identity  \eqref{poho} cannot hold, and so $\sigma_{2}(Q)\ge4,$ which implies $2\sigma_1(Q)-2\gamma_1(Q)+K_{12}\sigma_2(Q)< 2$. Using Lemma \ref{le2.2}-\ref{weakconcentration} and Theorem D, we get that $\tilde{u}_{2k}$ is concentrate and $\sigma_2(p)\in2\mathbb{N}$ for $p\in\mathfrak{S}_2$, which contradicts $\rho_{2k}\to\rho_2\notin4\pi\mathbb{N}$ again. So we get that
\[\mathfrak{S}_2=\emptyset.\] By using $\sigma_2(Q)=0$ and Pohozaev identity \eqref{poho}, we see that $\sigma_1(Q)=2(1+\gamma_1(Q))$.  From $\sigma_1(Q)=2$, we get that  $\gamma_1(Q)=0$, that is,
\[ Q\notin S_1.\]
Next, we shall follow the arguments in \cite{llwy} to derive the shadow system \eqref{1.1}.
Let \[\left(\begin{array}{ll}u_{1k}\\ u_{2k}\end{array}\right)
 =\textbf{K}\left(\begin{array}{ll}v_{1k}\\ v_{2k}\end{array}\right).
\]Then we see that
\begin{equation*}
\left\{\begin{array}{ll}
-\Delta v_{1k}=\rho_{1k}(\frac{h_1e^{K_{11}v_{1k}+K_{12}v_{2k}}}{\int_Mh_1e^{K_{11}v_{1k}+K_{12}v_{2k}}dv_g}-1)=\rho_{1k}(h_1e^{\tilde{u}_{1k}}-1),\\
-\Delta v_{2k}=\rho_{2k}(\frac{h_2e^{K_{22}v_{2k}+K_{21}v_{1k}}}{\int_Mh_2e^{K_{22}v_{2k}+K_{21}v_{1k}}dv_g}-1)=\rho_{2k}(h_2e^{\tilde{u}_{2k}}-1),\\ \int_M v_{1k}dv_g=\int_Mv_{2k}dv_g=0.
\end{array}\right.
\end{equation*}
Since $\tilde{u}_{2k}$ is uniformly bounded from above, Green representation formula and $L^p$ estimate imply that $v_{2k}$ is uniformly bounded in $M$ and converges to some
function $\frac12w$ in $C^{1,\alpha}(M)$  for $\alpha\in(0,1)$ as $k\to+\infty$.

 We define $\tilde{v}_{1k}=v_{1k}-\frac12\log\int_{M}\tilde{h}_ke^{2v_{1k}}dv_g,$ where $\tilde{h}_k=h_1e^{K_{12}v_{2k}}$. Then $\tilde{v}_{1k}$ satisfies
\begin{equation*}\left\{\begin{array}{ll}
\Delta \tilde{v}_{1k}+\rho_{1k}( \tilde{h}_ke^{2\tilde{v}_{1k}}-1)=0,\\ \int_M\tilde{h}_ke^{2\tilde{v}_{1k}}dv_g=1.\end{array}\right.
\end{equation*}
We see that $\tilde{v}_{1k}$ blows up at $Q$ as $k\to+\infty$.
 Since $\tilde{h}_k\to h_1e^{\frac{K_{12}w}{2}}$ in $C^{1,\alpha}(M)$ as $k\to+\infty$ and $Q\notin S_1$,    \cite[Theorem 0.2-0.3]{l}  and \cite{bt1}
  imply
\begin{equation}
\label{2.23}\left\{\begin{array}{ll}
\tilde{v}_{1k}\rightarrow-\infty\ \textrm{uniformly on any compact subset of}~M\setminus\{Q\}\ \mbox{as}\ k\to+\infty,\\
\rho_{1k}\frac{h_1e^{2v_{1k}+K_{12}v_{2k}}}{\int_Mh_1e^{2v_{1k}+K_{12}v_{2k}}dv_g}\rightarrow 4\pi\delta_Q,
\\ v_{1k}=\tilde{v}_{1k}-\int_M\tilde{v}_{1k}dv_g\rightarrow 4\pi G(x,Q)~\mathrm{in}~C^{2,\alpha}_{\textrm{loc}}(M\setminus\{Q\}).\end{array}\right.
\end{equation}From \cite[ESTIMATE B]{cl1}, we also get
\begin{equation}
\label{2.24}
\nabla \log(h_1e^{\frac{K_{12}w}{2}}) \mid_{x=Q}=0.
\end{equation}
Moreover, from \cite[Lemma 4.1]{cl1}, we can find some constant $c>0$, independent of $k$,
\begin{align}
\label{2.25}
\Big|2\nabla\tilde{v}_{1k}-\nabla\Big(\log\frac{e^{\lambda_{k}}}{\big(1+\frac{\rho_{1k}\tilde{h}_k(p^{(k)})e^{\lambda_{k}}}{4}|x-p^{(k)}|^2\big)^2}\Big)\Big|
<c~\mathrm{for}~|x-Q|<r_0,
\end{align}where $\tilde{v}_{1k}(p^{(k)})=\max_{B_{r_0}(Q)}\tilde{v}_{1k}$.
By   (\ref{2.25}) and standard  elliptic estimate, we get that
\begin{equation}
\label{2.26}
v_{2k}\rightarrow\frac12w \ \textrm{in}  \ C^{2,\alpha}(M)\ \textrm{for} \ \alpha\in(0,1)\  \textrm{as}  \ k\to+\infty,\end{equation}
where $w$ satisfies
\begin{equation}
\label{2.27}
\Delta w+2\rho_2(\frac{h_2e^{w+K_{21}4\pi G(x,Q)}}{\int_Mh_2e^{w+K_{21}4\pi G(x,Q)}dv_g}-1)=0.
\end{equation}
From \eqref{2.23}-\eqref{2.27}, we complete the proof of Theorem \ref{th1.1}.
  \hfill $\square$

\vspace{1cm}

\section{Compactness of solutions of shadow system}\label{section3}
  Let
$\rho_2\notin 4\pi\mathbb{N}$. We recall the following shadow system:
\begin{align*}
\quad\left\{\begin{array}{l}
\Delta w +2\rho_2\left(\frac{h_2e^{w +4\pi K_{21} G(x,Q)}}{\int_Mh_2e^{w +4\pi K_{21} G(x,Q )}}-1\right)=0,~\int_M w =0,\\
\nabla\left(\log(h_1e^{\frac{K_{12}}{2}w })\right)\Big|_{x=Q}=0,\ Q\notin S_1.
\end{array}\right.~~\quad\quad\quad\quad\quad\quad\quad\quad(1.1)
\end{align*}
  We note that   any solution $(w,Q)$ of \eqref{1.1} belongs to $\mathring{H}^1(M)\times [M\setminus  S_1 ].$
 Moreover, in the following proposition,  we  shall   prove  the compactness of the solutions of \eqref{1.1} in $\mathring{H}^1(M)\times [M\setminus  S_1 ]$.
 We recall

\medskip
 \noindent
  \textbf{Theorem 1.1.} \emph{Suppose} $\alpha_{p,i}\in\{1,2\}$ \emph{and} $\rho_2\notin4\pi\mathbb{N}$. \emph{Then there are constants} $C>0$ \emph{and} $\delta>0$ \emph{such that}
\emph{for any solution} $(w,Q)$ \emph{of} \eqref{1.1},
\[\|w\|_{C^1(M)}\le C \  \textrm{and}\ \textrm{dist}(Q,S_1)\ge\delta>0.\]

%\begin{proposition}
%\label{prop3.1}
%All the solutions $(w,Q)$ of \eqref{1.1} are  compact as a subset of $C^{2,\alpha}(M)\times [M\setminus  S_1 ],$ that is, there are constants $C, \delta>0$ such that
%$\|w\|_{C^1(M)}<C$, and $\mathrm{dist}(Q,S_1)>\delta$ for any solution $(w,Q)$ of \eqref{1.1}.
%\end{proposition}
\begin{proof}
For a solution $(w,Q)$ of \eqref{1.1}, we denote
$$\tilde{w} =w -\log\left(\int_Mh_2e^{w +4\pi K_{21} G(x,Q)}\right).$$
Then
\begin{equation}
\int_Mh_2e^{\tilde{w} +4\pi K_{21} G(x,Q)}=1.\label{3.1}
\end{equation}
Since $h_i(x)=h_i^*(x)e^{-4\pi\sum_{q\in S_i}\alpha_{q,i}G(x,q)}$ and $h_i^*(x)>0,~i=1,2,$ we can rewrite the system \eqref{1.1} to
\begin{align*}
\quad\quad\quad
\left\{\begin{array}{l}
\Delta \tilde{w} +2\rho_2\left(h_2e^{\tilde{w} +4\pi K_{21} G(x,Q )}-1\right)=0,\\
\nabla\Big(\log h_1^*+\frac{K_{12}}{2}w -4\pi\sum_{q\in S_1}\alpha_{q,1}G(x,q) \Big)\Big|_{x=Q }=0.
\end{array}\right.\quad\quad\quad\quad\quad\quad  (\tilde{S})
\end{align*}

 We claim that  there is a constant $c>0$ such that $\sup_M\tilde{w}\le c$ for any solution $(\tilde{w},Q)$ of the system $(\tilde{S})$.
To prove this claim, we argue by contradiction and suppose  that there is a sequence of solutions  $(\tilde{w}_k,Q_k)$ of the system $(\tilde{S})$ such that  $\tilde{w}_k$ blows up as $k\to+\infty$.
Next, we shall prove  it is impossible by the following steps.

\medskip
 \noindent
\textit{Step 1.} We claim that $Q_0\in S_2$, where $Q_0:=\lim_{k\to+\infty}Q_k$.
\medskip

\noindent  Let  \[\mathcal{B}=\{p\in M\  | \ \exists \{x_k\},\ x_k\to p,\ \tilde{w}_k(x_k)\to+\infty\ \textrm{as}\ k\to+\infty\},\]
and
\begin{equation*}
\sigma(p)=\lim_{\delta\to0}\lim_{k\to+\infty}\frac{1}{2\pi}\int_{B_{\delta}(p)}\rho_2h_2e^{\tilde{w}_k+4\pi K_{21} G(x,Q_k)}.
\end{equation*}
It is known that $|\mathcal{B}|<+\infty$. For example, see \cite{bm} or the  arguments of Lemma \ref{le2.1}.
We note that the singular set of the equation $(\tilde{S})$ is $S_2\cup\{Q_k\}$.
If $Q_0\notin S_2$, then $(\tilde{S})$ has no collapsing singular points. In this case, Theorem A implies no blowup for $\tilde{w}_k$ if $2\rho_2\notin 8\pi\mathbb{N}$, which yields a contradiction to the assumption. Thus we conclude that $Q_0\in S_2$.

\medskip
 \noindent
\textit{Step 2.}   We claim that  $\mathcal{B}=\{Q_0\}$.

\medskip
 \noindent
Let $r>0$ be a small constant satisfying  $B_r(p_i)\cap B_r(p_j)$ for any $p_i\neq p_j\in \mathcal{B}$. If $p_0\in\mathcal{B}\setminus \{Q_0\}\neq\emptyset$,  then the  Brezis-Merle Theorem \cite[Theorem 3]{bm}  implies
that $\tilde{w}_k\to-\infty$ in any compact subset of $B_{r}(p_0)\setminus\{p_0\}$, and then we have
$\tilde{w}_k\to-\infty$ in any compact subset of $M\setminus\mathcal{B}$.
Since $\alpha_{p,2}\in\mathbb{N}$ for any $p\in S_2$, we have the local mass $\sigma(p_0)$ for $p_0\in S_2\setminus\{Q_0\}$ is
an even positive integer by Theorem A (see also \cite{bt1}).  We note that $Q_0\notin\mathcal{B}$, then $\sigma(Q_0)=0$.  On the other hand,  if $Q_0\in\mathcal{B}$, then $\sigma(Q_0)\in 2\mathbb{N}$ by  Theorem \ref{thmeven}.
Hence,
  the sum of local masses  of all the blow up points is an even positive integer. As a consequence,  $2\rho_2\in8\pi\mathbb{N}$,
which yields a contradiction to $\rho_2\notin4\pi\mathbb{N}$.

\medskip
 \noindent\textit{Step 3.} If $\tilde{w}_k$ is concentrate, we   get a contradiction again  by using Theorem \ref{thmeven} and $\rho_2\notin 4\pi\mathbb{N}$ as in  \textit{Step 2}.
Therefore, $\tilde{w}_k$ is non-concentrate.
Then we claim that
\begin{equation}
\label{3.29}
\alpha_{Q_0,2}-K_{21}+1>\sigma(Q_0)\in 2\mathbb{N}.
\end{equation}
Since $\tilde{w}_k$ is non-concentrate,   there is a function $\tilde{w}_0$ satisfying
 $\tilde{w}_k\to \tilde{w}_0\ \ \textrm{in}\ \ C^2_{\textrm{loc}}(M\setminus\{Q_0\})$  as $k\to+\infty$, and
\begin{equation*}
 \Delta \tilde{w}_0  +2\rho_2(h_2e^{\tilde{w}_0(x)+4\pi K_{21}  G(x,Q_0)}-1)=-4\pi\sigma(Q_0)\delta_{Q_0}\ \textrm{on}\ M,
\end{equation*}
and \begin{equation}\label{L1}h_2e^{\tilde{w}_0(x)+4\pi K_{21}\pi G(x,Q_0)}\in L^1(M).\end{equation}
We recall that $h_2(x)=h_2^*(x)e^{-4\pi\sum_{p\in S_2}\alpha_{p,2}G(x,p)}$ and $h_2^*>0$.
Then by using $Q_0\in S_2,$ \eqref{L1}  and  the Green representation formula       as in Lemma \ref{weakconcentration},   we can obtain \eqref{3.29}.

\medskip

Before we proceed the next step, we make the  following preparation. Let
  \begin{equation}Q_k-Q_0=\varepsilon_k e_k,\  \varepsilon_k>0,\  |e_k|=1,\ \lim_{k\to+\infty}e_k=e\in\mathbb{R}^2.\label{3.30}\end{equation}Then $\lim_{k\to+\infty}\varepsilon_k=0$.   If there is a subsequence $Q_{k_l}$ such that $Q_{k_l}\equiv Q_0$, then   equation has no collapsing singularity which implies $\tilde{w}_{k_l}$ does not blow up by Theorem A. So we may assume that
   \[\varepsilon_k\neq0\ \textrm{ for all}\ \ k.\]
Let \begin{align}v_k(x)= \tilde{w}_k({\varepsilon_k}x+Q_0) +(2+2\alpha_{Q_0,2}-2K_{21})\ln{\varepsilon_k}.\label{definition_of_vk}\end{align} Then $v_k$ satisfies
\begin{align}\label{equationofvk}
 \Delta v_k +2\rho_2 (h_k({\varepsilon_k}x)e^{v_k(x)}|x|^{2\alpha_{Q_0,2}} |x-e_k |^{-2K_{21}}-\varepsilon_k^2)=0
\  \textrm{on}\ B_{\frac{r}{{\varepsilon_k}}}(0),
\end{align}
where
$$h_k(x-Q_0)={h}_2^*(x )e^{-4\pi\left(\sum_{p\in S_2\setminus\{Q_0\}}\alpha_{p,2} G(x,p )+\alpha_{Q_0,2}R(x,Q_0)-K_{21}R(x,Q_k)\right)},$$
and $R(x,p)$ denotes the regular part of the Green function $G(x,p).$  We note that there is a constant $c_0>0$ such that $h_k(x)\ge c_0>0$ on $B_r(0)$  for all $k$.

\medskip
 \noindent
\textit{Step 4.} We claim that  $v_k$ blows up at some finite points in $\mathbb{R}^2$.  Since the proof of this claim is long, we postpone it in  Lemma \ref{le3.04} below.

\medskip
 \noindent
\textit{Step 5.} In this step, we will determine the location of blow up points of $v_k$.
\medskip

Let $\mathcal{B}_{v}$ be the set of finite blow up points of $v_k$ such that
\begin{align}\label{defofBv}
\mathcal{B}_v=\left\{p\in B_{\frac{r}{{\varepsilon_k}}}(0)\  | \ \exists \{x_k\},\ x_k\to p,\ v_k(x_k)\to+\infty\ \textrm{as}\ k\to+\infty\right\}.
\end{align}By using \eqref{3.29} and the assumption $\alpha_{p,i}\in\{1,2\}$, we get the following possibilities (i)-(iii):
\begin{itemize}\item[(i)]  When $K_{21}=-1$ ($\textbf{A}_2$ case),
  we have $\sigma(Q_0)<4$ and    $\sigma(Q_0)=2$, which implies
  $|\mathcal{B}_v|=1$ and $\mathcal{B}_v\cap\{0,e\}=\emptyset$.
\item[(ii)]  When $K_{21}=-2$ ($\textbf{B}_2$ case),    we have $\sigma(Q_0)<5$ and   $\sigma(Q_0)\in\{2, 4\}$, which implies $|\mathcal{B}_v|=1,2$.
  If $v_k$ blows up at $e$, then  $\sigma(Q_0)\ge  2-2K_{21}=6$, which   contradicts $\sigma(Q_0)\in\{2, 4\}$.
 If $v_k$ blows up at $0$, then $\sigma(Q_0)\ge 2+2\alpha_{Q_0,2}\ge 4$ and thus $\sigma(Q_0)=4$, $\alpha_{Q_0,2}=1$. However, from \eqref{3.29}, we see that  $\alpha_{Q_0,2}-K_{21}+1= 4>\sigma(Q_0)=4$, which implies a contradiction.
 Thus $\mathcal{B}_v\cap\{0,e\}=\emptyset$.
\item[(iii)] When $K_{21}=-3$ ($\textbf{G}_2$ case),    we have $\sigma(Q_0)<6$ and   $\sigma(Q_0)\in\{2,4\}$, which implies $|\mathcal{B}_v|=1,2$.
  If $v_k$ blows up at $e$, then  $\sigma(Q_0)\ge 2-2K_{21}=8$, which   contradicts $\sigma(Q_0)\in\{2, 4\}$.
 If $\alpha_{Q_0,2}=2$ and $v_k$ blows up at $0$, then $\sigma(Q_0)\ge 2+2\alpha_{Q_0,2}=6$, which   contradicts $\sigma(Q_0)\in\{2, 4\}$.
If $\alpha_{Q_0,2}=1$, then it might be possible that  $v_k$ blows up at $0$ since \eqref{3.29}  holds in this case, that is, $\alpha_{Q_0,2}-K_{21}+1=5>\sigma(Q_0)\ge 2+2\alpha_{Q_0,2}= 4$.
 Thus $\mathcal{B}_v\cap\{0,e\}=\emptyset$ or $\mathcal{B}_v=\{0\}$.
\end{itemize}

\medskip

  In conclusion, we get either $\mathcal{B}_v\cap\{0,e\}=\emptyset$ or $\mathcal{B}_v=\{0\}$.  So  $v_k$ does not blow up at $e$.
  Moreover, we claim that if  $\mathcal{B}_v\cap\{0,e\}=\emptyset$, then \begin{align}\label{3.036}
\sum_{q_i\in\mathcal{B}_v}\frac{K_{21}e}{q_i-e}>0,
\end{align}
which will be used to yield a contradiction after $\nabla w_k(Q_k)$ is computed (see \eqref{3.330} below).
We postpone the proof of the claim \eqref{3.036} in Lemma \ref{lemmaforclaim} later.

\medskip
 \noindent
\textit{Step 6.} In this step, we will compute $\nabla w_k(Q_k)$ and derive a contradiction from the second equation in $(\tilde{S})$.
The computation $\nabla w_k(Q_k)$ depends on whether blow up occurs at the   singularity  $0$  or not.

\medskip

\noindent \textit{Case 1.} $\mathcal{B}_v\cap\{0,e\}=\emptyset$.
\medskip

By using Green's representation formula,  we see that as $k\to+\infty$,
\begin{equation}\label{3.21}\nabla v_k(x) \to  -\sum_{q_i\in \mathcal{B}_{v}}4 \frac{x-q_i}{|x-q_i|^2}\ \textrm{uniformly in}\ C^0_{\textrm{loc}}(B_{\frac{r}{{\varepsilon_k}}}(0)\setminus \mathcal{B}_{v} ).\end{equation}
Then we see that as $k\to+\infty$, \begin{equation}
\label{3.330}
{\varepsilon_k}\nabla w_k(Q_k)=\nabla v_k (e_k ) \to -\sum_{q_i\in \mathcal{B}_{v}}4\frac{e-q_i}{|e-q_i|^2},
\end{equation}where we used $\lim_{k\to+\infty}e_k=e$. We regard  $x\in B_{\frac{r}{{\varepsilon_k}}}(0)$ as a complex value $x=x_1+ix_2\in\mathbb{C}$ and denote its conjugate by $\bar{x}$.
The balance condition in $(\tilde{S})$  and \eqref{3.330} imply that
\begin{equation}
\begin{aligned}
\label{3.037}
0=~&\lim_{k\to+\infty}{\varepsilon_k}\left( \sum_{p\in S_1}2\alpha_{p,1}\frac{Q_k-p}{|Q_k-p|^2}+\frac{K_{12}}{2}\nabla w_k(Q_k) \right)\\
=~&2\alpha_0e-2 K_{12}\sum_{q_i\in\mathcal{B}_v}\frac{e-q_i}{|e-q_i|^2}=2\alpha_0\frac{1}{\bar{e}}-2 K_{12}\sum_{q_i\in\mathcal{B}_v}\frac{1}{\bar{e}-\bar{q_i}},
\end{aligned}
\end{equation}where $\alpha_0=\alpha_{Q_0,1}$ if $Q_0 \in S_1$ and
$\alpha_0=0$ if $Q_0 \notin S_1$.
In view of \eqref{3.036} and \eqref{3.037}, we get that
\begin{equation}
\begin{aligned}
\label{3.038}
0=\alpha_0- K_{12}\sum_{q_i\in\mathcal{B}_v}\frac{e}{e-q_i}>0,
\end{aligned}
\end{equation}
which implies a contradiction.

\medskip

\noindent \textit{Case 2.} $\mathcal{B}_v=\{0\}$.
\medskip

By using Green's representation formula,  we see that as $k\to+\infty$,
\begin{align}\label{gradient_conv}
{\varepsilon_k}\nabla w_k(Q_k)=\nabla v_k (e_k ) \to-4(\alpha_{Q_0,2}+1)e.
\end{align}
The balance condition in $(\tilde{S})$ gives
\begin{equation}
\begin{aligned}
\label{3.40}
0=&\lim_{k\to+\infty}{\varepsilon_k}\Big(\sum_{p\in S_1}2\alpha_{p,1}\frac{Q_k-p}{|Q_k-p|^2}+\frac{K_{12}}{2}\nabla w_k(Q_k)
 \Big)\\
=&~2\alpha_0e-2(\alpha_{Q_0,2}+1) K_{12}e,
\end{aligned}
\end{equation}where $\alpha_0=\alpha_{Q_0,1}$ if $Q_0 \in S_1$ and
$\alpha_0=0$ if $Q_0 \notin S_1$.
Since $2\alpha_0-2(\alpha_{Q_0,2}+1) K_{12}>0$, we must have $e=0$, which contradicts to $|e|=1$.

 \bigskip

Finally, from the above arguments, we conclude that $\tilde{w}_k$ cannot blow up.
Using Green's representation formula for  \eqref{1.1}, we see
\begin{equation}
\label{3.42}w(x)=2\int_M\rho_2 h_2(y)e^{\tilde{w}(y)+4\pi K_{21} G(y,Q_k)}G(x,y)dy.\end{equation}
Since there is a constant $c>0$ such that $\sup_M\tilde{w}\le c$ for any solution $(\tilde{w},Q)$ of the system $(\tilde{S})$,  we see from \eqref{3.42} that $w$ is uniformly bounded for  any solution $(w,Q)$ of the system \eqref{1.1}. By standard elliptic estimate and the balance condition in \eqref{1.1}, we can find some
constants $C, \delta>0$ such that
\begin{equation*}
\|w\|_{C^1(M)}\le C,
\ \ \textrm{and}\ \  \textrm{dist}(Q,S_1) >\delta>0, \end{equation*}for any solution $(w,Q)$ of the system \eqref{1.1}.
  Now we complete the proof of Theorem \ref{compactnessthm}.
\end{proof}
Now we are going to prove the claim \eqref{3.036}.
\begin{lemma}\label{lemmaforclaim}Let $v_k$ be a solution of \eqref{equationofvk}. We assume that  as $k\to+\infty$, $v_k$ blows up at   points $p\in\mathcal{B}_v$ (see \eqref{defofBv} for the definition of $\mathcal{B}_v$).   If $|\mathcal{B}_v|=1,2$ and  $\mathcal{B}_v\cap\{0,e\}=\emptyset$, then \[
\sum_{q_i\in\mathcal{B}_v}\frac{K_{21}e}{q_i-e}>0.\]
\end{lemma}\begin{proof}
If  $\mathcal{B}_v\cap\{0,e\}=\emptyset$, then the location of blow up points can be obtained by Pohozaev type identity (see \cite[ESTIMATE B]{cl1}): for any $q_i\in\mathcal{B}_v$, it holds that
\begin{equation}
\label{3.034}
2\sum_{q_j\in\mathcal{B}_v\setminus\{q_i\}}\frac{q_j-q_i}{|q_j-q_i|^2}+ \alpha_{Q_0,2}\frac{q_i}{|q_i|^2}- K_{21}\frac{q_i-e}{|q_i-e|^2}=0.
\end{equation}
We regard  $x\in B_{\frac{r}{{\varepsilon_k}}}(0)$ as a complex value $x=x_1+ix_2\in\mathbb{C}$.  Then  for any $q_i\in\mathcal{B}_v$,
\begin{equation*}
2\sum_{q_j\in\mathcal{B}_v\setminus\{q_i\}}\frac{q_i}{ q_j - {q_i}}+\alpha_{Q_0,2} - K_{21}\frac{q_i}{ {q_i}- {e}}=0.
\end{equation*}
Taking the summation  for $q_i\in\mathcal{B}_v$, we get
\begin{equation}
\begin{aligned}
\label{3.035}
0&=2\sum_{q_i\in\mathcal{B}_v}\sum_{q_j\in\mathcal{B}_v\setminus\{q_i\}}\frac{q_i}{ q_j-q_i }+ \alpha_{Q_0,2}|\mathcal{B}_v| -K_{21}\sum_{q_i\in\mathcal{B}_v}\frac{q_i}{q_i-e}\\
&=-|\mathcal{B}_v|(|\mathcal{B}_v|-1)+(\alpha_{Q_0,2}-K_{21})|\mathcal{B}_v| -K_{21}\sum_{q_i\in\mathcal{B}_v}\frac{e}{q_i-e}.
\end{aligned}
\end{equation}
Now  we  get
\begin{align*}
\sum_{q_i\in\mathcal{B}_v}\frac{K_{21}e}{q_i-e}=\left\{\begin{array}{ll}
 (\alpha_{Q_0,2}-K_{21} )>0\quad\quad\   \textrm{if}\ \ |\mathcal{B}_v|=1,\\
 2(\alpha_{Q_0,2}-K_{21}-1)>0\ \textrm{if}\ \ |\mathcal{B}_v|=2,
\end{array}\right.
\end{align*}
and complete the proof of Lemma \ref{lemmaforclaim}.
\end{proof}

\vspace{1cm}

\section{Proof of Theorem \ref{thmeven} and Lemma \ref{le3.04}}
In this section, we are going to  prove Theorem \ref{thmeven} and Lemma \ref{le3.04}.
%show that  if a sequence of solutions  $\tilde{w}_k$ of the first equation in $(\tilde{S})$ is non-concentrate, then  the scaled function   by \eqref{definition_of_vk} blows up.
We recall the following equation:
 \begin{equation*}
\quad\quad\quad\quad\quad\quad \Delta \hat{u}_k+2\rho_2    \hat{h} e^{\hat{u}_k}  =4\pi\sum_{p_{k_j}\in\hat{S}_k}\beta_{j} \delta_{p_{k_j}}\ \textrm{in}\ B_1(0),\quad\quad\quad\quad\quad\quad\quad\quad\ (1.9)
\end{equation*}  where $ \hat{h}>0$,  $|\hat{S}_k|$ is independent of $k$,    $\lim_{k\to+\infty}p_{k_j}=0$ for all $p_{k_j}\in \hat{S}_k$,  $p_{k_i}\neq p_{k_j}$ if $i\neq j$,  and  $\beta_{j}\in\mathbb{N}$. We assume that
\begin{align}\label{assumption}
 \left\{\begin{array}{l}(i)\ 0 \ \textrm{is the only blow up point:}\\   \max_{|x|\ge r}\hat{u}_k\le C(r),\ \sup_{B_1(0)}\hat{u}_k\to+\infty\ \textrm{as}\ k\to+\infty,\\  \\
(ii) \ \textrm{bounded oscillation:}\\  \sup_{x, y\in\partial B_1(0)}|\hat{u}_k(x)-\hat{u}_k(y)|\le c \ \textrm{for   some constant}\ c>0,\\  \\
(iii)\ \textrm{finite mass}:\\    \int_{B_1(0)}   \hat{h} e^{\hat{u}_k}\le C\ \textrm{for   some constant}\ C>0.
\end{array}\right.
\end{align}
 We denote the local mass at the blow up point $0$ by
 \begin{equation*}
  \sigma_0:=\lim_{\delta\to0}\lim_{k\to+\infty}\frac{1}{2\pi} \int_{B_{\delta}(0)}\rho_2\hat{h} e^{\hat{u}_k}.  \end{equation*}
  Clearly, $\sigma_0>0$ by (i).
 To prove Theorem \ref{thmeven}, we need to  consider the following equation: \begin{align}\label{entiresolutions}
 \left\{\begin{array}{l}
\Delta u +e^u=\sum_{i=1}^N 4\pi\alpha_i\delta_{p_i}\ \ \textrm{in}\ \ \R^2,\\    \\ \int_{\R^2}e^{u}dx<+\infty,
\end{array}\right.
\end{align}where $p_j$ are distinct points in $\R^2$ and $\alpha_i\in\mathbb{N}$.
The following  result in \cite{lwz} plays a crucial role in proving Theorem \ref{thmeven}.
\begin{theorem}\cite[Theorem 2.1]{lwz}\label{th3.3}
Let $u$ be a solution of \eqref{entiresolutions} and $\alpha_i\in\mathbb{N}$.
Then
%$u(x)=-\alpha\ln|x|+O(1)$ as  $|x|\to+\infty$,  and
\[\int_{\R^2} e^{u}dx=4\pi\Big(\sum_{i=1}^N\alpha_i+\frac{\alpha}{2}\Big)\in 8\pi\mathbb{N},\ \textrm{where}\ \alpha>2.\]
\end{theorem}
 Now we are going to prove Theorem \ref{thmeven}.

\medskip
 \noindent \textbf{Proof of Theorem \ref{thmeven}.}  If $|\hat{S}_k|=1$, then  there is no collapsing in the singular sources. If $\sigma_0\notin 2\mathbb{N}$, then Theorem A implies no blowup for
$\hat{u}_k$, which yields a contradiction to the assumption (see also \cite{bt1}). Thus,  Theorem \ref{thmeven} holds when  $|\hat{S}_k|=1$.

  From now on, we consider the case  $|\hat{S}_k|\ge2$.
  To prove Theorem \ref{thmeven} when $|\hat{S}_k|\ge2$, we will compare the contribution of the  masses from  two different regions $B_{r}(0)\setminus B_{\varepsilon_k R}(0)$ and $B_{\varepsilon_k R}(0)$, where  $0<{r}\ll1, R\gg1$ are fixed constants and $\lim_{k\to+\infty}\varepsilon_k=0$.  To do it,  we will apply  the Pohozaev identity for the equation $(1.9)$ in  the region $B_{r}(0)\setminus B_{\varepsilon_k R}(0)$.

 Let $\varepsilon_k=\max\{|p_{k_i}-p_{k_j}|\ |\ p_{k_j}\neq p_{k_i}\in\hat{S}_k\}=|p_{k_1}-p_{k_2}|.$
We denote  \begin{equation}
\label{definitionofhatwk}\hat{w}_k(x)= \hat{u}_k(x)-2\sum_{j=1}^{|\hat{S}_k|}\beta_j\ln|x-p_{k_j}|,\end{equation} and
 \begin{align}
\label{3.009}\hat{v}_k(x)= \hat{w}_k({\varepsilon_k}x+p_{k_1})   +(2+2\sum_{j=1}^{|\hat{S}_k|}\beta_j)\ln{\varepsilon_k}.\end{align} Then $\hat{v}_k$ satisfies
\begin{align}
\label{3.9}
 \Delta \hat{v}_k +2\rho_2  \hat{h}({\varepsilon_k}x+p_{k_1})e^{\hat{v}_k(x)} \prod_{j=1}^{|\hat{S}_k|}|x-z_{k,j}|^{2\beta_j} =0
\  \textrm{in}\ B_{\frac{1}{{\varepsilon_k}}}(0),
\end{align}where $z_{k,j}=\frac{p_{k_j}-p_{k_1}}{\varepsilon_k}$. By the definition of $\varepsilon_k>0$, we see that for each $i\in\{1,2,\cdots,  |\hat{S}_k|\}$, there is a point $z_i\in\mathbb{R}^2$ such that  $z_i=\lim_{k\to+\infty}z_{k,i}.$ Fix   constants $0<r\ll1$ and $R\gg1$. Multiplying \eqref{3.9} by $\nabla_x \hat{v}_k\cdot x$ and integrating over $ B_{\frac{r}{{\varepsilon_k}}}(0)\setminus B_{ R}(0)$,  we get that
\begin{align}\label{3.11}
&\int_{\partial B_{r}(p_{k_1})}\Big[\frac{(\nabla_y \hat{w}_k\cdot (y-p_{k_1}))^2}{|y-p_{k_1}|}-\frac{|\nabla_y \hat{w}_k|^2|y-p_{k_1}|}{2}\nonumber
\\&+2\rho_2 \hat{h} (y)e^{\hat{w}_k(y)}|y-p_{k_1}|\prod_{j=1}^{|\hat{S}_k|} |y- (p_{k_j}-p_{k_1})|^{2\beta_j}\Big]d\sigma_y\nonumber
\\&-\int_{\partial B_{ R}(0)}\Big[\frac{(\nabla_x \hat{v}_k\cdot x)^2}{|x|}-\frac{|\nabla_x \hat{v}_k|^2|x|}{2}\\&+2\rho_2\hat{h}({\varepsilon_k}x+p_{k_1})e^{\hat{v}_k(x)}|x| \prod_{j=1}^{|\hat{S}_k|}|x-z_{k,j}|^{2\beta_j}\Big]d\sigma_x\nonumber
\\=&\int_{B_{\frac{r}{{\varepsilon_k}}}(0)\setminus B_{ R}(0)} \Big[2\rho_2{\varepsilon_k}x\cdot\nabla_{{\varepsilon_k}x}\hat{h}({\varepsilon_k}x+p_{k_1}) e^{\hat{v}_k(x)} \prod_{j=1}^{|\hat{S}_k|}|x-z_{k,j}|^{2\beta_j}\nonumber
\\&+2\rho_2\hat{h}({\varepsilon_k}x+p_{k_1})e^{\hat{v}_k(x)}\prod_{i=1}^{|\hat{S}_k|}|x-z_{k,i}|^{2\beta_i}\sum_{j=1}^{|\hat{S}_k|} \frac{2\beta_j(x-z_{k,j})\cdot z_{k,j}}{ |x-z_{k,j} |^2} \nonumber
\\&+(2+2\sum_{j=1}^{|\hat{S}_k|}\beta_j)2\rho_2\hat{h}({\varepsilon_k}x+p_{k_1})e^{\hat{v}_k(x)} \prod_{i=1}^{|\hat{S}_k|}|x-z_{k,i}|^{2\beta_i}\Big]dx.\nonumber
\end{align}
Let \begin{equation*}  M_{\varepsilon_k}(r)=\frac{1}{2\pi}\int_{B_{\frac{r}{{\varepsilon_k}}}(0)} \rho_2\hat{h}({\varepsilon_k}x+p_{k_1})e^{\hat{v}_k(x)} \prod_{i=1}^{|\hat{S}_k|}|x-z_{k,i}|^{2\beta_i}dx,
 \end{equation*}  and
 \begin{equation*} m_{\varepsilon_k}(R)=\frac{1}{2\pi}\int_{B_{ R}(0)} \rho_2\hat{h}({\varepsilon_k}x+p_{k_1})e^{\hat{v}_k(x)} \prod_{i=1}^{|\hat{S}_k|}|x-z_{k,i}|^{2\beta_i}dx.
 \end{equation*}
 We   note that
 \begin{equation}\label{defofsigma_0}\sigma_0 =\lim_{r\to0}\lim_{k\to+\infty}M_{\varepsilon_k}(r).
 \end{equation}
 We denote
 \begin{equation}\label{defofm_0}m_0:=\lim_{R\to+\infty}\Big(\lim_{k\to+\infty}m_{\varepsilon_k}(R)\Big).
 \end{equation}

Now we claim that \begin{equation}\label{pohoidforlocalmass}
\begin{aligned}
\pi[(2\sigma_0)^2-(2m_0)^2]=(2+2\sum_{j=1}^{|\hat{S}_k|}\beta_j)4\pi[\sigma_0-m_0].
\end{aligned}
\end{equation}
To prove the claim \eqref{pohoidforlocalmass},   we need to estimate   $\nabla \hat{w}_k$ on $\partial  B_{r}(0)$ and $\nabla \hat{v}_k$ on  $\partial B_{R}(0)$ (see Lemma \ref{gradientlemma} below).
We remark that if there is  no collapsing of singularities, these estimates are well known. We include the proof here for the sake of the completeness.
   \begin{lemma}\label{gradientlemma}\begin{itemize}\item[(i)] $\nabla \hat{w}_k(x)\to   -2\sigma_0 \frac{x}{|x|^2} +\nabla \phi$  in  $C_{\textrm{loc}}(B_{1}(0)\setminus\{0\})$ as $k\to+\infty$, where $\phi\in C^1(B_{1}(0))$.
\item[(ii)] for any $\delta>0$, there is $R>0$ such that  \[\lim_{k\to+\infty}\nabla \hat{v}_k(x)=   -2(m_0+o(1)) \frac{x}{|x|^2}\ \textrm{ for}\     |x|\ge R,\ \textrm{ where}\  |o(1)|\le\delta.\] \end{itemize}
\end{lemma}
\begin{proof}We will prove Lemma \ref{gradientlemma} by  the following steps.

\medskip
 \noindent \textit{Step 1.} First, we will prove the estimation (i).
 We note that
  \begin{equation}\label{generaleq1}
\Delta \hat{w}_k+2\rho_2\hat{h}(x)\prod_{j=1}^{|\hat{S}_k|} |x-p_{k_j}|^{2\beta_j}e^{\hat{w}_k}  =0\ \textrm{in}\ B_1(0).
\end{equation}
\\
Let $G_1$ be the Green's function on $B_1(0)$. Since $\sup_{x, y\in\partial B_1(0)}|\hat{u}_k(x)-\hat{u}_k(y)|\le c$, we see that for any $x_1, x_2\in B_1(0)$,
\begin{equation}\begin{aligned}\label{green}
&\hat{w}_k(x_1)-\hat{w}_k(x_2)\\&=\int_{B_1(0)}(G_1(x_1,y)-G_1(x_2,y))2\rho_2\hat{h}(y)\prod_{j=1}^{|\hat{S}_k|} |y-p_{k_j}|^{2\beta_j}e^{\hat{w}_k}dy +O(1),\end{aligned}\end{equation}
as $k\to+\infty$.

Then, we divide our discussion into  two cases according to the behavior of $\hat{w}_k$ on $\partial B_r(0)$:

 \medskip
 \noindent Case 1.   $\hat{w}_k$ is non-concentrate.
 \medskip

By \eqref{assumption}, Green's representation formula  \eqref{green}   and elliptic estimate, we see that there is a function $\hat{w}$ satisfying
$\hat{w}_k\to \hat{w}\ \ \textrm{in}\ \ C^2_{\textrm{loc}}(B_{1}(0)\setminus\{0\})$  as $k\to+\infty$ and
\begin{equation*}
 \Delta \hat{w}  +2\rho_2 \hat{h}(x)|x|^{2\sum_{j=1}^{|\hat{S}_k|}\beta_j}   e^{\hat{w} (x)}  = -4\pi\sigma_0\delta_0\ \textrm{in}\ B_{1}(0).
\end{equation*}
We note that \begin{equation}\hat{h}(x)|x|^{2\sum_{j=1}^{|\hat{S}_k|}\beta_j}   e^{\hat{w} (x)}\in L^1(B_{1}(0)),\label{generalL1}\end{equation}
which implies
\begin{equation}\label{3.8}\sum_{j=1}^{|\hat{S}_k|}\beta_j+1>\sigma_0.\end{equation}
Let $\phi(x)=\hat{w}+2\sigma_0 \ln|x|$. Since $\Delta \phi\in L^1(B_{1}(0))$, by applying \cite[Theorem 1]{bm} as in Lemma \ref{weakconcentration}, we see that for any $\delta>0$, there is $r_{\delta}>0$ such that   $e^{(1+\delta)|\phi|}\in L^1(B_{r_\delta}(0))$.
By standard elliptic estimate and \eqref{3.8}, we get that  $\phi\in C^{1}(B_{1}(0))$. Then  we can get the estimation (i) when $\hat{w}_k$ is non-concentrate.

  \medskip
 \noindent Case 2.   $\hat{w}_k$ is  concentrate.
\medskip

Then $\hat{w}_k\to-\infty$ in $C^0_{\textrm{loc}}( B_{1}(0)\setminus\{0\})$. Fix a point $x_0\in B_{1}(0)\setminus\{0\}$.
Let  $g_k=\hat{w}_k-\hat{w}_k(x_0)$. By \eqref{assumption}, Green's representation formula   \eqref{green}  and standard elliptic estimate, we see that there is a function $g$ satisfying  $g_k\to g\ \ \textrm{in}\ \ C^2_{\textrm{loc}}(B_{1}(0)\setminus\{0\})$  as $k\to+\infty$ and
$
 \Delta g   =-4\pi\sigma_0\delta_0\ \textrm{in}\ B_{1}(0).
$ Then we can see  $g+2\sigma_0\ln|x|\in C^1(B_{1}(0))$ easily.
As a consequence, we get the estimation (i) when $\hat{w}_k$ is  concentrate.
\bigskip

 In the left,  we  shall consider the behavior of     $\nabla \hat{v}_k$ on  $\partial B_{R}(0)$ for any fixed constant $R\gg1$, independent of $k$.

\medskip
 \noindent \textit{Step 2.} To prove the estimation (ii), we consider the following three cases (a)-(c) according to the behavior of $\hat{v}_k$:

\medskip
 \noindent (a)   $\hat{v}_k$ blows up.
\medskip

Let $\mathcal{B}_{v}$ be   the set of finite  blow  up points of $\hat{v}_k$ such that
\begin{equation}\label{def_of_Bv}\mathcal{B}_{v}=\Big\{p\in B_{\frac{r}{{\varepsilon_k}}}(0)\  | \ \exists \{x_k\},\ x_k\to p,\ \hat{v}_k(x_k)\to+\infty\ \textrm{as}\ k\to+\infty\Big\}.\end{equation} We denote the local mass at $p\in \mathcal{B}_{v}$ by
\begin{equation}\label{def_of_sigmav}\sigma_v(p)=\lim_{\delta\to0}\lim_{k\to+\infty}\frac{1}{2\pi}\int_{B_{\delta}(p)}\rho_2 \hat{h}({\varepsilon_k}x+p_{k_1})e^{\hat{v}_k(x)} \prod_{i=1}^{|\hat{S}_k|}|x-z_{k,i}|^{2\beta_i}dx.
\end{equation}

In view of \eqref{assumption} and Green representation formula \eqref{green}, it is not difficult to show that  $\hat{v}_k$ also has bounded oscillation  near each blow up point, i.e., for any blow up point  $p\in \mathcal{B}_v$   and a small fixed constant $r>0$, there is a constant $c_r>0$ such that  \begin{equation}\label{boundedoscillation}\sup_{x,y\in\partial B_r(p)}|\hat{v}_k(x)-\hat{v}_k(y)|\le c_r.\end{equation}

Since there might be  collapsing singularities $z_{k,j}$ in \eqref{3.9} as $k\to +\infty$, we need to consider the following two possibilities (a-i)-(a-ii):

\medskip
 \noindent (a-i) If  $\hat{v}_k\to-\infty$ in $C^0_{\textrm{loc}}(B_{\frac{r}{{\varepsilon_k}}}(0)\setminus \mathcal{B}_{v})$, then we see that \[m_0=\sum_{p\in \mathcal{B}_{v}} \sigma_v(p).\]
Moreover,  \eqref{assumption} and Green's representation formula \eqref{green} imply that as $k\to+\infty$,
\begin{equation}\label{3.21n}\nabla \hat{v}_k(x)\to  -\sum_{p\in \mathcal{B}_{v}}2\sigma_v(p)\frac{x-p}{|x-p|^2} \  \ \textrm{in}\ C^0_{\textrm{loc}}(B_{\frac{r}{{\varepsilon_k}}}(0)\setminus \mathcal{B}_{v} ). \end{equation}
So  the estimation (ii) holds for the case (a-i).

\medskip
 \noindent (a-ii) If there is a function $\hat{v}_0$ such that $\hat{v}_k\to \hat{v}_0$ in $C^1_{\textrm{loc}}(B_{\frac{r}{{\varepsilon_k}}}(0)\setminus \mathcal{B}_{v})$, then $\hat{v}_0$ satisfies
\begin{align*}
 \Delta \hat{v}_0  +2\rho_2 \hat{h}(0) e^{\hat{v}_0 (x)} \prod_{j=1}^{|\hat{S}_k|}|x-z_j |^{2\beta_j}=-4\pi\sum_{p\in \mathcal{B}_v}\sigma_v(p)\delta_p\   \textrm{in}\ \R^2.
\end{align*}
Since $e^{\hat{v}_0 (x)} \prod_{j=1}^{|\hat{S}_k|}|x-z_j |^{2\beta_j}\in L^1(\mathbb{R}^2)$, we see that if $p\neq z_j$, then  $\sigma_v(p)<1$, and if $p= z_j$, then  $\sigma_v(z_j)<1+\beta_j$. Here we note that if $p\neq z_j$, then  $\sigma_v(p)=0$  since there is no collapsing singularities near $p\neq z_j$ and thus $\sigma_v(p)\in2\mathbb{N}\cup\{0\}$ by \cite{bt1}.       By standard potential analysis (see \cite[Lemma 1.2]{cli}), we see that\[\hat{v}_0(x)=-(2\sum\beta_j+\alpha)\ln|x|+O(1)\ \textrm{ as}\  |x|\to+\infty,\ \textrm{ where}\  \alpha>2,\] and
\begin{align*}
 \int_{\R^2}2\rho_2 \hat{h}(0) e^{\hat{v}_0 (x)} \prod_{j=1}^{|\hat{S}_k|}|x-z_j |^{2\beta_j}=4\pi\Big(\sum_{j=1}^{|\hat{S}_k|} (\beta_j-\sigma_v(z_j))+\frac{\alpha}{2}\Big).
\end{align*}
Thus we get that \[m_0= \sum_{j=1}^{|\hat{S}_k|}\sigma_v(z_j)+\Big(\sum_{j=1}^{|\hat{S}_k|} (\beta_j-\sigma_v(z_j))+\frac{\alpha}{2}\Big)=\Big(\sum_{j=1}^{|\hat{S}_k|} \beta_j+\frac{\alpha}{2}\Big).\]
Moreover,   in view of \cite[Lemma 1.3]{cli}, we obtain as $|x|\to+\infty$,
\begin{equation*}\begin{aligned}
 &\nabla\Big(\hat{v}_0(x)+2\sum_{p\in \mathcal{B}_v}\sigma_v(p)\ln|x-p|\Big) \\&= -2\Big(\sum_{j=1}^{|\hat{S}_k|} (\beta_j-\sigma_v(z_j))+\frac{\alpha}{2}+o(1)\Big)\frac{x}{|x|^2} \ \textrm{as}\  |x|\to+\infty,
\end{aligned}\end{equation*}which implies
\begin{equation*}\begin{aligned}
 &\nabla \hat{v}_0(x)   = -2\Big(\sum_{j=1}^{|\hat{S}_k|} \beta_j +\frac{\alpha}{2}+o(1)\Big)\frac{x}{|x|^2}=-2 (m_0+o(1))\frac{x}{|x|^2} \   \textrm{as}\  |x|\to+\infty.
\end{aligned}\end{equation*}
Since $\hat{v}_k\to \hat{v}_0$ in $C^1_{\textrm{loc}}(B_{\frac{r}{{\varepsilon_k}}}(0)\setminus \mathcal{B}_{v})$, the estimation (ii) holds for the case  (a-ii).

\medskip
 \noindent (b) $\hat{v}_k\to-\infty$ uniformly on compact subsets of $B_{\frac{r}{{\varepsilon_k}}}(0)$ as $k\to+\infty$.\\
In this case, we can see that $m_0=0.$
By using \eqref{assumption} and Green's representation formula \eqref{green}, we get   that   $\lim_{k\to+\infty} \nabla \hat{v}_k =0$ { in} $C_{\textrm{loc}}^0(B_{\frac{r}{{\varepsilon_k}}}(0))$ { as} $k\to+\infty.$
So   the estimation (ii) holds for the case (b).

\medskip
 \noindent (c) $\hat{v}_k$ is locally uniformly bounded in $L^\infty_{\textrm{loc}}(B_{\frac{r}{{\varepsilon_k}}}(0))$ as $k\to+\infty$.\\
For this case, we can conclude  $\hat{v}_k$ converges to a function $\hat{v}$ in $C^2_{\textrm{loc}}(B_{\frac{r}{{\varepsilon_k}}}(0))$ as $k\to+\infty$, and
\begin{align*}
 \Delta \hat{v}  +2\rho_2 \hat{h}(0) e^{\hat{v}(x)} \prod_{j=1}^{|\hat{S}_k|}|x-z_j |^{2\beta_j}=0\ \textrm{in}\ \R^2.
\end{align*}
By standard potential analysis (see \cite[Lemma 1.2]{cli}), we see that \[\hat{v}(x)=-(2\sum\beta_j+\alpha)\ln|x|+O(1)\ \textrm{ as}\  |x|\to+\infty\  \textrm{where}\  \alpha>2,\] and
\begin{align*}
 \int_{\R^2}2\rho_2 \hat{h}(0) e^{\hat{v}(x)} \prod_{j=1}^{|\hat{S}_k|}|x-z_j |^{2\beta_j}=4\pi\Big(\sum_{j=1}^{|\hat{S}_k|} \beta_j+\frac{\alpha}{2}\Big).
\end{align*}Then
 $m_0= \sum_{j=1}^{|\hat{S}_k|} \beta_j+\frac{\alpha}{2}.$
Moreover, in view of \cite[Lemma 1.3]{cli}, we obtain
\begin{align}
 \nabla\hat{v}(x)= -2\Big(\sum_{j=1}^{|\hat{S}_k|} \beta_j+\frac{\alpha}{2} +o(1)\Big)\frac{x}{|x|^2}=-2(m_0+o(1))\frac{x}{|x|^2} \ \textrm{as}\  |x|\to+\infty.
\end{align}
So the estimation (ii) holds for the case (c).

\medskip

From the discussion in \textit{Step 1}-\textit{Step 2}, we  complete the proof of Lemma \ref{gradientlemma}.
\end{proof}

\medskip
 \noindent \textbf{Proof of \eqref{pohoidforlocalmass}.}
In the proof of Lemma \ref{gradientlemma}, we also get
\begin{itemize}
\item[(i)] Fix a constant $0<r\ll1$, independent of $k$. Then on $\partial B_r(0)$, we have either $\lim_{k\to+\infty}e^{\hat{w}_k}r^{2+2\sum_{j=1}^{|\hat{S}_k|}\beta_j}=0$, or $\lim_{k\to+\infty}e^{\hat{w}_k}r^{2+2\sum_{j=1}^{|\hat{S}_k|}\beta_j}= e^{\phi}r^{2+2\sum_{j=1}^{|\hat{S}_k|}\beta_j-2\sigma_0}$  {where}   $\sum_{j=1}^{|\hat{S}_k|}\beta_j+1>\sigma_0.$
    Thus  we  have \[\lim_{r\to0}\lim_{k\to+\infty}e^{\hat{w}_k}r^{2+2\sum_{j=1}^{|\hat{S}_k|}\beta_j}=0.\]
\item[(ii)] Fix a constant $R\gg1$, independent of $k$. Then on $\partial B_R(0)$, we have either $\lim_{k\to+\infty}e^{\hat{v}_k}R^{2+2\sum_{j=1}^{|\hat{S}_k|}\beta_j}=0$, or $\lim_{k\to+\infty}e^{\hat{v}_k}R^{2+2\sum_{j=1}^{|\hat{S}_k|}\beta_j}=O(R^{2-\alpha})$ where $\alpha>2.$
     Thus  we  have \[\lim_{R\to+\infty}\lim_{k\to+\infty}e^{\hat{v}_k}R^{2+2\sum_{j=1}^{|\hat{S}_k|}\beta_j}=0.\]
\end{itemize}
Then by letting $k\to+\infty$ in \eqref{3.11}, we get that
\[
\pi[(2\sigma_0)^2-(2m_0)^2]=(2+2\sum_{j=1}^{|\hat{S}_k|}\beta_j)4\pi[\sigma_0-m_0]+o(1)\ \textrm{as}\ r\to0,\ R\to+\infty.
\]
So we prove the claim \eqref{pohoidforlocalmass}.\hfill$\Box$

\medskip

 We see that  \eqref{pohoidforlocalmass} implies
\begin{equation}
\begin{aligned}\label{resultpohoid}
\sigma_0=m_0,\ \ \textrm{or}\ \ \sigma_0=2+ 2\sum_{j=1}^{|\hat{S}_k|}\beta_j-m_0.
\end{aligned}
\end{equation}

In  next lemma, we will show that $m_0\in 2\mathbb{N}\cup\{0\}$ and then Theorem \ref{thmeven}   follows  immediately.
\begin{lemma}\label{evenm0}$m_0\in2\mathbb{N}\cup\{0\}$.\end{lemma}\begin{proof}
We note  that     if $|\hat{S}_k|=2$,   then  after scaling, there is no collapsing  singularities in \eqref{3.9} since $z_1=0$ and $|z_2|=1$.
On the other hand, if $|\hat{S}_k|>2$, then  even after scaling, there might be  collapsing  singularities $\{z_{k,i}\}$ in \eqref{3.9} as $k\to+\infty$.  Therefore,  we will prove  Theorem \ref{thmeven} by   mathematical induction method on $|\hat{S}_k|$.  First, we let  $|\hat{S}_k|=2$.

\medskip
 \noindent {\em{Step 1.}}  We assume that   $\hat{u}_k$ is a solution of $(1.9)$ satisfying \eqref{assumption} with
 $|\hat{S}_k|=2$.   Let $\varepsilon_k=\max\{|p_{k_i}-p_{k_j}|\ |\ p_{k_j}\neq p_{k_i}\in\hat{S}_k\}=|p_{k_1}-p_{k_2}|$ and
 \begin{align}\label{hatofvk}
\hat{v}_k(x)= \hat{u}_k({\varepsilon_k}x+p_{k_1}) -2\sum_{j} \beta_j\ln|{\varepsilon_k}x+p_{k_1}-p_{k_j}|  +(2+2\sum_{j}\beta_j)\ln{\varepsilon_k}.\end{align} Then $\hat{v}_k$ satisfies
\begin{align}\label{twosingular}
 \Delta \hat{v}_k +2\rho_2  \hat{h}({\varepsilon_k}x+p_{k_1})e^{\hat{v}_k(x)} \prod_{j}|x-z_{k,j}|^{2\beta_j} =0
\  \textrm{in}\ B_{\frac{1}{{\varepsilon_k}}}(0),
\end{align}where $z_{k,j}=\frac{p_{k_j}-p_{k_1}}{\varepsilon_k}$. Let $z_i=\lim_{k\to+\infty}z_{k,i}.$  Then $z_1=0$ and $|z_2|=1$. Since $z_1\neq z_2$ and $|\hat{S}_k|=2$, the equation \eqref{twosingular} has no collapsing singularities. Thus Brezis-Merle Theorem \cite[Theorem 3]{bm} implies there are three possible behaviors  (a)-(c) of $\hat{v}_k$:

\medskip
 \noindent (a) $\hat{v}_k$ blows up. Let  $\mathcal{B}_{v}$ be   the set of finite  blow  up points of $\hat{v}_k$  defined by \eqref{def_of_Bv}, and $\sigma_v(p)$ be the local mass at $p\in \mathcal{B}_{v}$  defined by \eqref{def_of_sigmav}.
Since  the equation \eqref{twosingular} has no collapsing singularities,  Brezis-Merle Theorem \cite[Theorem 3]{bm} implies $\hat{v}_k\to-\infty$ in $C^0_{\textrm{loc}}(B_{\frac{r}{{\varepsilon_k}}}(0)\setminus \mathcal{B}_{v})$. From  $\beta_j\in\mathbb{N}$ and Theorem A (see also \cite{bt1}),  we   get
\begin{equation*}\sigma_v(p)\in2\mathbb{N}\  \ \textrm{for} \ \ p\in \mathcal{B}_{v}.\end{equation*}
Then $m_0=\sum_{p\in \mathcal{B}_{v}} \sigma_v(p)\in 2\mathbb{N},$
which implies $\sigma_0 \in 2\mathbb{N}$ by \eqref{resultpohoid}.

\medskip
 \noindent (b) $\hat{v}_k\to-\infty$ uniformly on compact subsets of $B_{\frac{r}{{\varepsilon_k}}}(0)$ as $k\to+\infty$. Then $m_0=0.$

\medskip
 \noindent (c) $\hat{v}_k$ is locally uniformly bounded in $L^\infty_{\textrm{loc}}(B_{\frac{r}{{\varepsilon_k}}}(0))$ as $k\to+\infty$. Then  $\hat{v}_k$ converges to a function $\hat{v}$ in $C^2_{\textrm{loc}}(B_{\frac{r}{{\varepsilon_k}}}(0))$ as $k\to+\infty$,  where
\begin{align*}
 \Delta \hat{v}  +2\rho_2 \hat{h}(0) e^{\hat{v}(x)} \prod_j|x-z_j |^{2\beta_j}=0\ \textrm{in}\ \R^2,
\end{align*} and \[m_0=\frac{1}{2\pi}\int_{\R^2}\rho_2 \hat{h}(0) e^{\hat{v}(x)} \prod_j|x-z_j |^{2\beta_j}.\]
Since $\beta_j\in\mathbb{N}$,  Theorem \ref{th3.3} implies
\begin{align}\label{3.18}
 \int_{\R^2}2\rho_2 \hat{h}(0) e^{\hat{v}(x)} \prod_j|x-z_j |^{2\beta_j}=4\pi\Big(\sum_j \beta_j+\frac{\alpha}{2}\Big)\in8\pi\mathbb{N},\ \alpha>2,
\end{align}which implies $m_0=\Big(\sum_j \beta_j+\frac{\alpha}{2}\Big)\in 2\mathbb{N}.$

 \medskip

 Therefore, Lemma \ref{evenm0} is proved for  $|\hat{S}_k|=2$.

 \medskip
 \noindent {\em{Step 2.}}  We assume that that Lemma \ref{evenm0} holds if $|\hat{S}_k|\le n$, and suppose that
   $\hat{u}_k$ is a solution of $(1.9)$ satisfying \eqref{assumption} with $|\hat{S}_k|=n+1$. We do the same scaling as in the first step, and set $\hat{v}_k$ by \eqref{hatofvk}, which also satisfies \eqref{twosingular}.
If  $\hat{v}_k$ does not blow up, we can obtain $m_0\in2\mathbb{N}\cup\{0\}$ by using the same arguments for the case (b) and (c) in the first step.

In the left,  we consider the case  $\hat{v}_k$ blows up.
Let  $\mathcal{B}_{v}$ be   the set of finite  blow  up points of $\hat{v}_k$, and $\sigma_v(p)$ be
 the local mass at $p\in \mathcal{B}_{v}$.

 Note that  as in \eqref{boundedoscillation},  $\hat{v}_k$ also has bounded oscillation  near each blow up point.  Hence \eqref{assumption} holds at any blow up point of $\hat{v}_k$.

   Let $z_i=\lim_{k\to+\infty}z_{k,i}$.  From $z_1\neq z_2$,  we see that the number of collapsing singular points $z_{k,i}$ in \eqref{twosingular} is at most $n$.
  From our assumption,  Theorem \ref{thmeven} holds when the number of collapsing singularities $\le n$.
  So  the local mass $\sigma_v(p)$ at $p$ satisfies
\begin{equation}\label{evenblowup}\sigma_v(p)\in2\mathbb{N}\ \ \textrm{for}\ \ p\in \mathcal{B}_v.\end{equation}
Now we need to consider the following two cases (i)-(ii):

\medskip
 \noindent (i) If  $\hat{v}_k\to-\infty$ in $C^0_{\textrm{loc}}(B_{\frac{r}{{\varepsilon_k}}}(0)\setminus \mathcal{B}_{v})$, then $m_0=\sum_{p\in\mathcal{B}_v}\sigma_v(p)\in2\mathbb{N}$.

\medskip
 \noindent (ii) If there is a function $\hat{v}_0$ such that $\hat{v}_k\to \hat{v}_0$ in $C^0_{\textrm{loc}}(B_{\frac{r}{{\varepsilon_k}}}(0)\setminus \mathcal{B}_{v})$, then $\hat{v}_0$ satisfies
\begin{align*}
 \Delta \hat{v}_0  +2\rho_2 \hat{h}(0) e^{\hat{v}_0 (x)} \prod_{j=1}^{n+1}|x-z_j |^{2\beta_j}=-4\pi\sum_{p\in \mathcal{B}_v}\sigma_v(p)\delta_p\   \textrm{in}\ \R^2.
\end{align*}
The fact $e^{\hat{v}_0 (x)} \prod_{j=1}^{n+1}|x-z_j |^{2\beta_j}\in L^1(\mathbb{R}^2)$ implies that if $p\neq z_j$, then  $\sigma_v(p)<1$ and thus $\sigma_v(p)=0$ from \eqref{evenblowup}, and if $p= z_j$, then  $\sigma_v(z_j)<1+\beta_j$. Hence $\beta_j-\sigma_v(z_j)\in\mathbb{N}\cup\{0\}$.

Since $\beta_j, \sigma_v(z_j) \in\mathbb{N}$,   Theorem \ref{th3.3} implies
\begin{equation}\begin{aligned}\label{convergencemasss}
 &\int_{\R^2}2\rho_2 \hat{h}(0) e^{\hat{v}_0 (x)} \prod_{j=1}^{n+1}|x-z_j |^{2\beta_j}\\&=4\pi\Big(\sum_{j=1}^{n+1} (\beta_j-\sigma_v(z_j))+\frac{\alpha}{2}\Big)\in8\pi\mathbb{N},\ \alpha>2.
\end{aligned}\end{equation}Then by \eqref{evenblowup} and \eqref{convergencemasss}, we   get
$m_0=\sum_{j=1}^{n+1}  \sigma_v(z_j) + \Big(\sum_{j=1}^{n+1} (\beta_j-\sigma_v(z_j))+\frac{\alpha}{2}\Big)=  \sum_{j=1}^{n+1} \beta_j+\frac{\alpha}{2} \in 2\mathbb{N}.
$
Thus we complete the proof of Lemma \ref{evenm0}.
\end{proof}

For the proof of Theorem \ref{compactnessthm},  we also need to show that if $\hat{u}_k$ is non-concentrate, then    after the scaling as in \eqref{hatofvk},  the scaled function  $\hat{v}_k$      blows up as $k\to+\infty$.
Now we have the following lemma.
\begin{lemma}\label{le3.04}   Let $\hat{u}_k$ be a solution of $(1.9)$ satisfying \eqref{assumption}. We  set $\hat{v}_k$ by \eqref{hatofvk}, which   satisfies \eqref{twosingular}.
  If $\hat{u}_k$ is non-concentrate, then $\hat{v}_k$     blows up as $k\to+\infty$.
 \end{lemma}
\begin{proof}If $\hat{u}_k$ is non-concentrate, we have \eqref{3.8}, that is, $1+\sum_j\beta_j>\sigma_0$.\\
If the scaled function $\hat{v}_k$  does not   blows up as $k\to+\infty$, then we have the following two cases (i)-(ii):

\medskip
 \noindent
(i) if  $\hat{v}_k\to-\infty$ uniformly on compact subsets of $B_{\frac{r}{{\varepsilon_k}}}(0)$ as $k\to+\infty$, then    $m_0=0.$
By \eqref{resultpohoid} and $\sigma_0>0$, we get that
$
 \sigma_0=2+2\sum_j\beta_j\in 2\mathbb{N},
$   which contradicts  \eqref{3.8}.

\medskip
 \noindent
(ii) If  $\hat{v}_k$ is locally uniformly bounded in $L^\infty_{\textrm{loc}}(B_{\frac{r}{{\varepsilon_k}}}(0))$ as $k\to+\infty$, then   \textbf{step 1}-(c) in the proof of Lemma \ref{evenm0} implies
$m_0=\Big(\sum_j \beta_j+\frac{\alpha}{2}\Big)$, where $\alpha>2$. So $\sigma_0\ge m_0>\sum_j \beta_j+1$, which contradicts \eqref{3.8}.

 Thus  we complete the proof of Lemma \ref{le3.04}.
\end{proof}

\vspace{1cm}

\section{The topological degree of shadow system}
We recall the following shadow system:
\begin{align*}
\left\{\begin{array}{l}
\Delta w +2\rho_2\left(\frac{h_2e^{w +4\pi K_{21} G(x,Q )}}{\int_Mh_2e^{w +4\pi K_{21} G(x,Q )}}-1\right)=0,~\int_M w =0,\\
\nabla\left(\log(h_1e^{\frac{K_{12}}{2}w })\right)\Big|_{x=Q}=0,\ Q\notin S_1.
\end{array}\right.\quad\quad\quad\quad\quad\quad\quad\quad\   (1.1)
\end{align*} In this section, we are going to compute the topological degree of \eqref{1.1} for $\rho_2\notin 4\pi\mathbb{N}$.
As we discussed in the introduction, to compute the degree of (\ref{1.1}), we need to consider the following deformation:
\begin{align*}
\ \left\{\begin{array}{l}
\Delta w_t+2\rho_2\left(\frac{h_2e^{w_t+4\pi K_{21} G(x,Q_t)}}{\int_Mh_2e^{w_t+4\pi K_{21} G(x,Q_t)}}-1\right)=0,\ \ \int_M w_t=0,\\
\nabla\Big(\log h_1^*-4\pi\sum_{p\in S_1}\alpha_{p,1}G(x,p)+ \frac{t}{2}K_{12}w_t\\-4\pi(1-t)\sum_{p\in S_2\setminus S_1}G(x,p)\Big)\Big|_{x=Q_t}=0,
\end{array}\right.\quad\quad\quad\quad\quad\quad\  (1.10)_t
\end{align*} where $t\in[0,1)$.
We note that for fixed $t\in [0,1)$, any solution $(w_t,Q_t)$ of $(1.10)_t$ belongs to $\mathring{H}^1(M)\times [M\setminus (S_1\cup S_2)].$

\begin{proposition} For each fixed $t\in[0,1)$, there are $C_t, \delta_t>0$  such that for all solutions $(w_t,Q_t)$ of $(1.10)_t$ satisfies
\begin{equation}\label{03.1}
\|w_t\|_{C^1(M)}<C_t,~\mathrm{dist}(Q_t,S_1\cup S_2)>\delta_t>0.
\end{equation}\end{proposition}
\begin{proof}Suppose that there is  a sequence of solutions $(w_{t,k},Q_{t,k})$ of $(1.10)_t$ such that $w_{t,k}$ blows up as $k\to+\infty$. We note that the  coefficients of Green functions in $(1.10)_t$ have the same sign since $\alpha_{p,1}>0$   and $(1-t)>0$. Then, following  the same arguments  in the \textit{Step 6} in the proof of  Theorem \ref{compactnessthm}, we get a contradiction by noting  the same sign of   $\nabla w_{t,k}(Q_{t,k})$ and  $\nabla G(Q_{t,k},p)$, $p\in S_1\cup S_2$ (see \eqref{3.038} and  \eqref{3.40}).
So $w_{t,k}$ is uniformly bounded. Then    \eqref{03.1} follows from  the balance condition in $(1.10)_t$.\end{proof}

Moreover, we have the following    compactness result  for the solutions of   $(1.10)_t$ for  any $t\in[0,1)$.
\begin{proposition}
\label{prop4.01}
There are constants $C, \delta>0$ such that any solution  $(w_t,Q_t)$ of $(1.10)_t$ for $t\in[0,1)$ satisfies
 \begin{equation}\label{compact}
 \|w_t\|_{C^1(M)}<C,~\mathrm{dist}(Q_t,S_1)>\delta>0.
\end{equation}
\end{proposition}

\begin{proof}Since $t\in[0,1)$, we  can get $\|w_t\|_{C^1(M)}$ is uniformly bounded for all $t\in[0,1)$ as in the proof of \eqref{03.1}. On the other hand,
in view of $\lim_{t\to 1^-}(1-t)=0$,  it might be possible that the balance condition $(1.10)_t$  holds even though $Q_t$  converges to a point in  $S_2\setminus S_1$ as $t\to1^-$. So we conclude that
any solution  $(w_t,Q_t)$ of $(1.10)_t$ for $t\in[0,1)$  satisfies \eqref{compact}.
\end{proof}
In the following proposition, we shall meet   two possibilities  according to the behavior of $Q_t$ as $t\to 1^-$.
\begin{proposition}
\label{prop4.1}
Let  $(w_t,Q_t)$ be  a  family of solutions  of $(1.10)_t$. Then we have either

\noindent (i)   $Q_t\to Q\in M\setminus[S_1\cup S_2]$  and $w_t\to w_Q$ as $t\to1^-$, where $(w_Q,Q)$ satisfies \eqref{1.1},

or
\medskip

\noindent (ii)   $Q_t\to Q\in S_2\setminus S_1$ and $w_t\to w_Q$ as $t\to1^-$, where  $(w_Q,Q)$ satisfies
\begin{align}
&\Delta w_Q+2\rho_2\left(\frac{h_2e^{w_Q+4\pi K_{21} G(x,Q)}}{\int_Mh_2e^{w_Q+4\pi K_{21} G(x,Q)}}-1\right)=0,  \label{4.1}\\
  &\nabla\left(\log h_1+\frac12K_{12}w_Q\right)\Big|_{x=Q}=\lambda_{w_Q}e_{w_Q},\ Q\notin S_1,\label{4.1b}
\end{align}
where $(\lambda_{w_Q},e_{w_Q})$ satisfies $ \lambda_{w_Q} \ge0, \ e_{w_Q}\in \mathbb{R}^2, |e_{w_Q}|=1$, and
$$4\pi(1-t)\nabla\Big(\sum_{p\in S_2\setminus S_1}G(x,p)\Big)\Big|_{x=Q_t}\to\lambda_{w_Q} e_{w_Q}~\textrm{ as}~t\to1^-.$$
\end{proposition}
\begin{proof}
  Proposition \ref{prop4.1} is a simple consequence of  Proposition \ref{prop4.01}.
\end{proof}

\noindent

By using transversality theorem  (for example, see \cite[Theorem 4.1]{llwy}), we can always choose a  function   $h_2$ satisfying the following condition  (C1):

\medskip
 \noindent
 (C1) For any point $Q\in S_2\setminus S_1$, all the solutions  $w_Q$  {of  }  \eqref{4.1}  { are non-degenerate}.

\medskip
 \noindent
We note that the set $S_2\setminus S_1$ is fixed and $|S_2\setminus S_1|<+\infty$.
By (C1), there are finitely many solutions $w_Q$ of \eqref{4.1}. Moreover, any solution $w_Q$ of \eqref{4.1} is independent of $h_1$. So we can always choose a   function   $h_1$ satisfying the following condition  (C2):

\medskip
 \noindent
  (C2)  For any point $Q\in S_2\setminus S_1$, $\nabla\left(\log h_1+\frac12K_{12}w_Q\right)\Big|_{x=Q}\neq0$.

\medskip
 \noindent
Throughout the rest of this section, we choose  $h_1$ and $h_2$ such that (C1) and (C2) hold.

\medskip

For any point $Q\in S_2\setminus S_1$, we consider
\begin{align}
\label{4.3}
\Delta w_Q +2\rho_2\left(\frac{h_2e^{w_Q+4\pi K_{21} G(x,Q)}}{\int_Mh_2e^{w_Q+4\pi K_{21} G(x,Q)}}-1\right)=0.
\end{align}
For each non-degenerate solution $w_Q$ of \eqref{4.3}, let $(\lambda_{w_Q},e_{w_Q})$ be the pair satisfying
\begin{align}
\label{4.4}
\left\{\begin{array}{l}
\nabla\left(\log h_1+\frac{1}{2}K_{12}w_Q\right)\Big|_{x=Q}=\lambda_{w_Q}e_{w_Q},\\
\lambda_{w_Q}>0, \ e_{w_Q}\in \mathbb{R}^2,\ |e_{w_Q}|=1.
\end{array}\right.
\end{align}
For some fixed $\mathfrak{p}> 2$, we define $\|\phi\|_*=\|\phi\|_{W^{2,\mathfrak{p}}(M)}$. Let $\Gamma_{w_Q,t}$ be
\begin{equation}
\label{4.5}
\begin{aligned}
\Gamma_{w_Q,t}:=\Big\{(w_t,Q_t)~\Big|~w_t=w_Q+\phi_t,\ \int_{M}\phi_t=0,\  \|\phi_t\|_{*}\le M_0(1-t),\\
Q_t=Q-\frac{2(1-t)}{\lambda_t}e_t,\  e_t=\frac{Q-Q_t}{|Q-Q_t|},\  |Q_t-Q|=\frac{2(1-t)}{\lambda_t},\ \\
\frac{1}{2}\le\frac{\lambda_t}{\lambda_{w_Q}}\le 2,\ |\lambda_te_t-\lambda_{w_Q}e_{w_Q}|\le M_0(1-t) \Big\},
\end{aligned}
\end{equation}
where $M_0>0$ is a large number, which is determined later.
\medskip

\noindent We have the following a priori estimate for the family of solutions $(w_t,Q_t)$ of $(1.10)_t$, if $(w_t,Q_t)$ satisfies Proposition \ref{prop4.1}-(ii).

\begin{lemma}
\label{le4.2} Assume (C1) and (C2).
Let $(w_t,Q_t)$ be a family of solutions of $(1.10)_t$ satisfying Proposition \ref{prop4.1}-(ii). Then there are some function $w_Q$ and constant $\varepsilon>0$ such that $w_Q$ is a solution of \eqref{4.3} and if $|t-1|<\varepsilon$, then $(w_t,Q_t)\in \Gamma_{w_Q,t}$.
\end{lemma}

\begin{proof}
Since $(w_t,Q_t)$ satisfies Proposition \ref{prop4.1}-(ii), we find that     $Q_t\to Q\in S_2\setminus S_1$ and $w_t\to w_Q$ as $t\to1^-$, where
\begin{align}
\label{4.6}
\left\{\begin{array}{l}
\Delta w_Q+2\rho_2\left(\frac{h_2e^{w_Q+4\pi K_{21} G(x,Q)}}{\int_Mh_2e^{w_Q+4\pi K_{21} G(x,Q)}}-1\right)=0,\\
\nabla\left(\log h_1+\frac12K_{12}w_Q\right)\Big|_{x=Q}=\lambda_{w_Q} e_{w_Q},
\end{array}\right.
\end{align}
and  $\lambda_{w_Q}\in\mathbb{R}\setminus\{0\},\ e_{w_Q}\in \mathbb{R}^2,~|e_{w_Q}|=1$, and $$4\pi(1-t)\nabla\Big(\sum_{p\in S_2\setminus S_1}G(x,p)\Big)\Big|_{x=Q_t}\to\lambda_{w_Q} e_{w_Q}\ \ \textrm{ as} \ \ t\to1^-.$$
Let $w_t=w_Q+\phi_t$. Then $\phi_t$ satisfies $\int_M\phi_t =0$ and
\begin{align*}
\mathcal{L}\phi_t:=~&\Delta \phi_t+2\rho_2\frac{h_2e^{w_Q+4\pi K_{21} G(x,Q)}}{\int_Mh_2e^{w_Q+4\pi K_{21} G(x,Q)}}\phi_t\\
&-2\rho_2\frac{h_2e^{w_Q+4\pi K_{21} G(x,Q)}}{\left(\int_Mh_2e^{w_Q+4\pi K_{21} G(x,Q)}\right)^2}\int_Mh_2e^{w_Q+4\pi K_{21} G(x,Q)}\phi_tdx\\
=~&\mathbb{I}_1+\mathbb{I}_2,
\end{align*}
where
\begin{align*}
\mathbb{I}_1=&-2\rho_2\left(\frac{h_2e^{w_t+4\pi K_{21} G(x,Q)}}{\int_Mh_2e^{w_t+4\pi K_{21} G(x,Q)}}-\frac{h_2e^{w_Q+4\pi K_{21} G(x,Q)}}{\int_Mh_2e^{w_Q+4\pi K_{21} G(x,Q)}}\right)\\
&+2\rho_2\frac{h_2e^{w_Q+4\pi K_{21} G(x,Q)}}{\int_Mh_2e^{w_Q+4\pi K_{21} G(x,Q)}}\phi_t\\
&-2\rho_2\frac{h_2e^{w_Q+4\pi K_{21} G(x,Q)}}{\left(\int_Mh_2e^{w_Q+4\pi K_{21} G(x,Q)}\right)^2}\int_Mh_2e^{w_Q+4\pi K_{21} G(x,Q)}\phi_t.
\end{align*}
and
\begin{align*}
\mathbb{I}_2=-2\rho_2\left(\frac{h_2e^{w_t+4\pi K_{21} G(x,Q_t)}}{\int_Mh_2e^{w_t+4\pi K_{21} G(x,Q_t)}}-\frac{h_2e^{w_t+4\pi K_{21} G(x,Q)}}{\int_Mh_2e^{w_t+4\pi K_{21} G(x,Q)}}\right).
\end{align*}
We can easily see that $\mathbb{I}_1=O(\|\phi_t\|_*^2)$. For $\mathbb{I}_2$, we note that
\begin{align*}
e^{4\pi K_{21} G(x,Q_t)}=e^{4\pi K_{21} G(x,Q)}+O(|Q_t-Q|).
\end{align*}
Thus we get $\mathbb{I}_2=O(|Q_t-Q|)$. By (C1), the solution $w_Q$ of the first equation in \eqref{4.6} is non-degenerate. Using the non-degeneracy of $w_Q$, we can find a constant $c_0>0$, independent of $t\in[0,1)$, such that
\begin{align}
\label{4.7}
\|\phi_t\|_{*}\leq c_0|Q_t-Q|.
\end{align}
Furthermore, from the balance condition in $(1.10)_t$ and \eqref{4.6}, we have
\begin{equation*}
\begin{aligned}
0=~&\nabla(\log h_1+\frac12K_{12}{w_Q})\Big|_{x=Q_t}-\nabla(\log h_1+\frac12K_{12}{w_Q})\Big|_{x=Q}\\&-\frac12(1-t)K_{12}\nabla w_Q(Q_t)\\
&-4\pi(1-t)\sum_{p\in S_2\setminus(S_1\cup\{Q\})}\nabla G(Q_t,p)-4\pi(1-t)\nabla R(Q_t,Q)\\&+2(1-t)\frac{Q_t-Q}{|Q_t-Q|^2}\\
&-\frac12tK_{12}\left(\nabla w_Q(Q_t)-\nabla w_t(Q_t)\right)+\lambda_{w_Q}e_{w_Q},
\end{aligned}
\end{equation*}where $R(x,p)$ denotes the regular part of the Green function $G(x,p)$.
By \eqref{4.7}, we can find constants $c_1, c_2>0$ which are independent of $t$, such that
\begin{equation}
\label{4.8}
\begin{aligned}
\Big|\lambda_{w_Q} -2\frac{(1-t) }{|Q_t-Q| }\Big|&\le\Big|\lambda_{w_Q}  e_{w_Q}+2(1-t)\frac{Q_t-Q}{|Q_t-Q|^2}\Big|\\&\le c_1(1-t)+c_2|Q_t-Q|,
\end{aligned}
\end{equation}
which implies
\begin{equation*}
|Q_t-Q|\le \frac{1}{\lambda_{w_Q}}\Big(2(1-t)+c_1(1-t)|Q_t-Q|+c_2|Q_t-Q|^2\Big).
\end{equation*}
As a consequence, we get a constant $c_3>0$, independent of $t$, satisfying
\begin{equation}
\label{4.9}
|Q_t-Q|\le c_3(1-t).
\end{equation}
Let $e_t=\frac{Q-Q_t}{|Q-Q_t|}$, $\lambda_t=\frac{2(1-t)}{|Q_t-Q|}$ and choose $M_0>\max\{2c_0c_3, 2c_1 +2c_2c_3\}$. Then from \eqref{4.7}-\eqref{4.9}, we get that
\begin{equation}
\label{4.10}
\|\phi_t\|_{*}< M_0(1-t) \ \textrm{and}\ |\lambda_te_t-\lambda_{w_Q} e_{w_Q}|< M_0(1-t).
\end{equation}
Moreover, we also notice that
\begin{equation*}
\lambda_{w_Q}\Big|\frac{\lambda_t }{\lambda_{w_Q}}-1\Big|\le \lambda_{w_Q}\Big|\frac{\lambda_t }{\lambda_{w_Q}}e_t-e_{w_Q}\Big|
=|\lambda_te_t-\lambda_{w_Q} e_{w_Q}|\le M_0(1-t).
\end{equation*}
Thus, we get $\Big|\frac{\lambda_t }{\lambda_{w_Q}}-1\Big|\le \frac{M_0(1-t)}{\lambda_{w_Q} },$ which implies
\begin{equation}
\label{4.11}
\frac{1}{2}<\frac{\lambda_t}{\lambda_{w_Q} }< 2 \ \textrm{ when}\  t \textrm{ is close to}\  1.
\end{equation}
In view of \eqref{4.10}-\eqref{4.11}, we prove that there is $\varepsilon>0$ such that if $|t-1|<\varepsilon$, then
$(w_t,Q_t)\in\Gamma_{w_Q,t}$, where $w_Q$ is a solution of \eqref{4.3}. Now we complete the proof of Lemma \ref{le4.2}.
\end{proof}

Conversely,  for any $Q\in S_2\setminus S_1$ and any non-degenerate solution $w_Q$ of \eqref{4.3},  we shall construct a sequence of solutions $(w_t,Q_t)\in \Gamma_{w_Q,t}$ of $(1.10)_t$ for $t$ close to $1$ such that
$$w_t\to w_Q, \ \textrm{ and}\  4\pi(1-t)\nabla_x\sum_{p\in S_2\setminus S_1}G(x,p)\Big|_{x=Q_t}\to\lambda_{w_Q} e_{w_Q}\ \textrm{ as}\  t\to 1.$$
For $\rho_2\in (4\pi j, 4\pi (j+1))$,  any $Q\in S_2\setminus S_1$ and any non-degenerate solution $w_Q$ of \eqref{4.3},  let    $d_{j}(w_Q,Q)$  be  the degree contributed by a solution $w_Q$  of \eqref{4.3}, and $d_{\Gamma_{w_Q,t},j}$ denote the degree contributed by the solutions of $(1.10)_t$ from the set $\Gamma_{w_Q,t}.$ Then, we have
the following lemma.
\begin{lemma}
\label{le4.3}Assume (C1) and (C2).
For any $Q\in S_2\setminus S_1$, let $w_Q$ be a non-degenerate solution of \eqref{4.3} and $(\lambda_{w_Q},e_{w_Q})$ satisfy \eqref{4.4}. Then there exists a constant $\varepsilon>0$ such that if $|t-1|<\varepsilon$, we can find a sequence of solutions $(w_t,Q_t)\in \Gamma_{w_Q,t}$ of $(1.10)_t$ such that
\begin{equation}
\label{4.12}
Q_t\to Q, \ w_t\to w_Q,\  4\pi(1-t)\sum_{p\in S_2\setminus S_1}\nabla G(x,p)\Big|_{x=Q_t}\to\lambda_{w_Q} e_{w_Q}~\textrm{as}~t\to1^-.
\end{equation}Moreover, we have  $d_{\Gamma_{w_Q,t},j}=-d_{j}(w_Q,Q)$.
\end{lemma}
\begin{proof}
Let
\begin{align*}
T_0(w)=\Delta^{-1} \left[2\rho_2\left(\frac{h_2e^{w+4\pi K_{21} G(x,Q)}}{\int_Mh_2e^{w+4\pi K_{21} G(x,Q)}}-1\right)\right],
\end{align*}
and
\begin{align*}
T_1(w,z)=\Delta^{-1}\left[2\rho_2\left(\frac{h_2e^{w+4\pi K_{21} G(x,z)}}{\int_Mh_2e^{w+4\pi K_{21} G(x,z)}}-1\right)\right].
\end{align*}
We consider the following deformation $\Phi_{s,t}=(\Phi_{s,t}^1,\Phi_{s,t}^2)$ for $0\le s\le 1$, where
\begin{align*}
\Phi_{s,t}^1(w,z)=(1-s)\left(w+T_0(w)\right)+s\left(w+T_1(w,z)\right),
\end{align*}
and
\begin{align*}
\Phi_{s,t}^2(w,z)=~&(1-s)\Big(\lambda_{w_Q}e_{w_Q}+2(1-t)\frac{x-Q}{|x-Q|^2}\Big)\Big|_{x=z}\\
&+s\nabla\Big(\log h_1e^{\frac{t}{2}K_{12}w}-4\pi(1-t)\sum_{p\in S_2\setminus S_1}G(x,p)\Big) \Big|_{x=z}.
\end{align*}
We claim that there exists $\varepsilon>0$ such that for any $t$ satisfying $|t-1|<\varepsilon$, we have $\Phi_{s,t}\neq 0$ on $\partial\Gamma_{w_Q,t}$ for all $0\le s \le 1$. Suppose this claim is not true and there is $s\in[0,1]$ such that $\Phi_{s,t}(w_t,Q_t)=0$ for some $(w_t,Q_t)=(w_Q+\phi_t,Q-\frac{2(1-t)}{\lambda_t}e_t)\in\partial\Gamma_{w_Q,t}$.

Since $\Phi_{s,t}^1(w_t,Q_t)=0$, using the relation $w_t=w_Q+\phi_t$ and \eqref{4.3}, we have
\begin{align*}
0=~&(1-s)\left[\Delta w_t+ 2\rho_2\left(\frac{h_2e^{w_t+4\pi K_{21} G(x,Q)}}{\int_Mh_2e^{w_t+4\pi K_{21} G(x,Q)}}-1\right)\right]\\
&+s\left[\Delta w_t+2\rho_2\left(\frac{h_2e^{w_t+4\pi K_{21} G(x,Q_t)}}{\int_Mh_2e^{w_t+4\pi K_{21} G(x,Q_t)}}-1\right)\right]\\
=~&\Delta \phi_t+2\rho_2(1-s)\left(\frac{h_2e^{w_t+4\pi K_{21} G(x,Q)}}{\int_Mh_2e^{w_t+4\pi K_{21} G(x,Q)}}
-\frac{h_2e^{w_Q+4\pi K_{21} G(x,Q)}}{\int_Mh_2e^{w_Q+4\pi K_{21} G(x,Q)}}\right)\\
&+2\rho_2s\left(\frac{h_2e^{w_t+4\pi K_{21} G(x,Q_t)}}{\int_Mh_2e^{w_t+4\pi K_{21} G(x,Q_t)}}
-\frac{h_2e^{w_Q+4\pi K_{21} G(x,Q)}}{\int_Mh_2e^{w_Q+4\pi K_{21} G(x,Q)}}\right).
\end{align*}
%Then, we get
%\begin{align*}
%\mathcal{L}\phi_t=&~\Delta \phi_t+2\rho_2\frac{h_2e^{w_Q+4\pi K_{21} G(x,Q)}}{\int_Mh_2e^{w_Q+4\pi K_{21} G(x,Q)}}\phi_t\\
%&-2\rho_2\frac{h_2e^{w_Q+4\pi K_{21} G(x,Q)}}{\left(\int_Mh_2e^{w_Q+4\pi K_{21} G(x,Q)}\right)^2}\int_Mh_2e^{w_Q+4\pi K_{21} G(x,Q)}\phi_tdx\\
%=&~\mathbb{I}_3+\mathbb{I}_4,
%\end{align*}
%where
%\begin{align*}
%\mathbb{I}_3=&-2\rho_2\left(\frac{h_2e^{w_t+4\pi K_{21} G(x,Q)}}{\int_Mh_2e^{w_t+4\pi K_{21} G(x,Q)}}-\frac{h_2e^{w_Q+4\pi K_{21} G(x,Q)}}{\int_Mh_2e^{w_Q+4\pi K_{21} G(x,Q)}}\right)\\
%&+2\rho_2\frac{h_2e^{w_Q+4\pi K_{21} G(x,Q)}}{\int_Mh_2e^{w_Q+4\pi K_{21} G(x,Q)}}\phi_t\\
%&-2\rho_2\frac{h_2e^{w_Q+4\pi K_{21} G(x,Q)}}{\left(\int_Mh_2e^{w_Q+4\pi K_{21} G(x,Q)}\right)^2}\int_Mh_2e^{w_Q+4\pi K_{21} G(x,Q)}\phi_t.
%\end{align*}
%and
%\begin{align*}
%\mathbb{I}_4=-2\rho_2s\left(\frac{h_2e^{w_t+4\pi K_{21} G(x,Q_t)}}{\int_Mh_2e^{w_t+4\pi K_{21} G(x,Q_t)}}
%-\frac{h_2e^{w_t+4\pi K_{21} G(x,Q)}}{\int_Mh_2e^{w_t+4\pi K_{21} G(x,Q)}}\right).
%\end{align*}
%It is easy to see that $\mathbb{I}_3=O(\|\phi_t\|_*^2)$. On the other hand, we note
%\begin{align*}
%e^{4\pi K_{21} G(x,Q_t)}=e^{4\pi K_{21} G(x,Q)}+O(|Q-Q_t|).
%\end{align*}
%Thus, $\mathbb{I}_4=O(|Q_t-Q|)$.
As in the proof of Lemma \ref{le4.2}, by using the non-degeneracy of $w_Q$ to (\ref{4.3}) and $(w_t,Q_t)\in\partial\Gamma_{w_Q,t}$, we have a constant $c_0>0$, independent of $t$, satisfying
\begin{align}
\label{4.14}
\|\phi_t\|_{*}\leq c_0|Q_t-Q|\le\frac{2c_0(1-t)}{\lambda_t}\le \frac{4c_0(1-t)}{\lambda_{w_Q}}.
\end{align}
Furthermore, using $\Phi_{s,t}^2(w_t,Q_t)=0$ and (\ref{4.4}), we have
\begin{equation}\begin{aligned}\label{aboveeq}
0=&~(1-s)\left(\lambda_{w_Q}e_{w_Q}+2(1-t)\frac{Q_t-Q}{|Q_t-Q|^2}\right)\\
&+s\nabla\Big(\log h_1+\frac{t}{2}K_{12}w_t-4\pi(1-t)\sum_{p\in S_2\setminus S_1}G(x,p)\Big)(Q_t)\\
=&~\lambda_{w_Q}e_{w_Q}+2(1-t)\frac{Q_t-Q}{|Q_t-Q|^2}-\frac12stK_{12}(\nabla w_Q(Q_t)-\nabla w_t(Q_t))\\
&+s\Big(\nabla(\log h_1+\frac{1}{2}K_{12}w_Q)(Q_t)-\nabla(\log h_1+\frac12K_{12}w_Q)(Q)\Big)\\
&-\frac{1}{2}s(1-t)K_{12}\nabla w_Q(Q_t)\\&-4\pi s(1-t)\Big(\sum_{p\in S_2\setminus(S_1\cup\{Q\})}\nabla G(Q_t,p)+\nabla R(Q_t,Q)\Big).
\end{aligned} \end{equation}
Using \eqref{4.14},  $(w_t,\phi_t)\in\partial\Gamma_{w_Q,t}$ and \eqref{aboveeq}, we have
\begin{equation*}
\begin{aligned}
 \Big|\lambda_{w_Q}e_{w_Q}+2(1-t)\frac{Q_t-Q}{|Q_t-Q|^2}\Big| &=|\lambda_{w_Q}e_{w_Q}-\lambda_te_t|\le c_1(1-t)+c_2|Q_t-Q|\\&=c_1(1-t)+c_2\frac{2(1-t)}{\lambda_{w_Q}}\frac{\lambda_{w_Q}}{\lambda_t}\\&\le c_1(1-t)+\frac{4c_2}{\lambda_{w_Q}}(1-t),
\end{aligned}
\end{equation*}
for some constants $c_1, c_2>0$ which are independent of $t$. By choosing $M_0>\max\Big\{\frac{8c_0}{\lambda_{w_Q}}, 2c_1 +\frac{8c_2}{\lambda_{w_Q}}\Big\}$, we get that $\|\phi_t\|_{*}< M_0(1-t)$ and
$|\lambda_te_t-\lambda_{w_Q}e_{w_Q}|< M_0(1-t)$. As a consequence, we have
\begin{align*}
\lambda_{w_Q}\Big|\frac{\lambda_t}{\lambda_{w_Q}}-1\Big|\le \lambda_{w_Q}\Big|\frac{\lambda_t}{\lambda_{w_Q}}e_t-e_{w_Q}\Big|
=|\lambda_{w_Q}e_{w_Q}-\lambda_te_t|< M_0(1-t).
\end{align*}
Then we get $
\Big|\frac{\lambda_t}{\lambda_{w_Q}}-1\Big|< \frac{M_0(1-t)}{\lambda_{w_Q}},
$
which implies $\frac{1}{2}<\frac{\lambda_t}{\lambda_{w_Q}}< 2$ when $t$ is close to $1$. Therefore, we prove the claim that there is $\varepsilon>0$ such that if $|1-t|<\varepsilon$, then $\Phi_{s,t}\neq 0$ on $\partial\Gamma_{w_Q,t}$ for all $0\le s \le 1$.
\medskip

So we get that
\begin{equation}\label{4.16}d_{\Gamma_{w_Q,t},j}=\textrm{deg}(\Phi_{1,t},0,\Gamma_{w_Q,t})=\textrm{deg}(\Phi_{0,t},0,\Gamma_{w_Q,t})\ \ \textrm{if}\ \ |1-t|<\varepsilon.\end{equation}
When $\Phi_{0,t}(w_t,Q_t)=0$, i.e.,
\begin{align}
\label{4.17}
\left\{\begin{array}{l}
0=\Delta\Phi_{0,t}^1(w_t,Q_t)=\Delta w_t+2\rho_2\left(\frac{h_2e^{w_t+4\pi K_{21} G(x,Q)}}{\int_Mh_2e^{w_t+4\pi K_{21} G(x,Q)}}-1\right),\\
0=\Phi_{0,t}^2(w_t,Q_t)=\lambda_{w_Q}e_{w_Q}+2(1-t)\frac{Q_t-Q}{|Q_t-Q|^2}.
\end{array}\right.
\end{align}
Then we get that  $(w_t,Q_t)=\Big(w_Q,Q-\frac{2(1-t)}{\lambda_{w_Q}}e_{w_Q}\Big)\in \Gamma_{w_Q,t}$ by using the non-degeneracy of $w_Q$.

Next, we shall compute the term $\textrm{deg}(\Phi_{0,t},0,\Gamma_{w_Q,t})$. Let $Q_t=(Q_t^1,Q_t^2)$ and $Q=(Q^1,Q^2)$. Then we can see that
\begin{align*}
&\nabla_{Q_t} \Phi_{0,t}^2(w_t,Q_t)\\&=\frac{2(1-t)}{|Q_t-Q|^4}
\left[\begin{array}{ll}(Q_t^2-Q^2)^2-(Q_t^1-Q^1)^2\    &\   -2(Q_t^1-Q^1)(Q_t^2-Q^2)  \\
-2(Q_t^1-Q^1)(Q_t^2-Q^2)\   &\   (Q_t^1-Q^1)^2-(Q_t^2-Q^2)^2\\
\end{array}\right].
\end{align*}
We note that
$$\textrm{tr} \Big[\nabla_{Q_t} \Phi_{0,t}^2(w_t,Q_t)\Big]=0  \ \textrm{ and}\ \ \textrm{det} \Big[\nabla_{Q_t} \Phi_{0,t}^2(w_t,Q_t)\Big]<0.$$
Then the number of negative eigenvalues of $ \nabla_{Q_t} \Phi_{0,t}^2(w_t,Q_t)$ is one. So we deduce that the degree of the second equation in \eqref{4.17} is $-1$. Since \eqref{4.17} is decoupled system, the topological degree is the product of the topological degree of the two equations in \eqref{4.17}. Therefore, for any point $Q\in S_2\setminus S_1,$ we get from \eqref{4.16} that  if $|1-t|<\varepsilon$, then
\begin{equation}\label{4.18}d_{\Gamma_{w_Q,t},j}=\textrm{deg}(\Phi_{1,t},0,\Gamma_{w_Q,t})=\textrm{deg}(\Phi_{0,t},0,\Gamma_{w_Q,t})=-d_{j}(w_Q,Q),\end{equation}
 where   $d_{j}(w_Q,Q)$  is the degree contributed by a solution $w_Q$  of \eqref{4.3}.
 Since  all the solutions of (\ref{4.3}) are non-degenerate (see (C1)),  we can get $d_{\Gamma_{w_Q,t},j}=-d_{j}(w_Q,Q)\neq0$. Thus, for any point $Q\in S_2\setminus S_1$ and any non-degenerate solution $w_Q$ of (\ref{4.3}), we can construct a family of solutions $(w_t,Q_t)$ in $\Gamma_{w_Q,t}$ which verifies (\ref{4.12}) as $t$ is sufficiently close to $1$.\end{proof}
 For any $Q\in S_2\setminus S_1$ and $\rho_2\in (4\pi j, 4\pi (j+1))$, let $d_{j}(Q)$ denote the topological degree  of the equation (\ref{4.3}). We recall $d_{\Gamma_{w_Q,t},j}$ denotes the degree contributed by the solutions of $(1.10)_t$ in  the set $\Gamma_{w_Q,t}.$
Then we have the following result. \begin{lemma}\label{le4.3.1} There is a constant $\varepsilon>0$ such that
 \begin{equation}\label{4.13}
d_{j}(Q)=-\sum_{w_Q}d_{\Gamma_{w_Q,t},j}~\textrm{for}~|t-1|<\varepsilon.
\end{equation}\end{lemma}
 \begin{proof}For fixed $Q\in S_2\setminus S_1$, we consider the identity \eqref{4.18}. Let us  take the summation of $d_{\Gamma_{w_Q,t},j}$ and $d_{j}(w_Q,Q)$ with respect to all the solutions  $w_Q$ of \eqref{4.3}. Then we get
$$d_{j}(Q)=-\sum_{w_Q}d_{\Gamma_{w_Q,t},j},$$
and prove  Lemma \ref{le4.3.1}.
\end{proof}

\medskip
 \noindent
\textbf{Proof of Theorem \ref{th1.2}.} For $\rho_2\in(4\pi j, 4\pi (j+1))$,
let $d_{s,j}(t)$ be the topological degree of $(1.10)_t$ for $t\in[0,1)$.
At $t=0$, the system $(1.10)_0$ becomes the following decoupled system:
\begin{align}
\label{4.19}
\left\{\begin{array}{l}
\Delta w_0+2\rho_2\left(\frac{h_2e^{w_0+4\pi K_{21} G(x,Q_0)}}{\int_Mh_2e^{w_0+4\pi K_{21} G(x,Q_0)}}-1\right)=0,\\
\nabla\left(\log h_1^*-\sum_{p\in S_1}4\pi\alpha_{p,1}G(x,p)-4\pi\sum_{p\in S_2\setminus S_1}G(x,p)\right)\Big|_{x=Q_0}=0.
\end{array}\right.
\end{align}
From the balance condition in (\ref{4.19}), we can easily see that $Q_0\notin S_1\cup S_2$. Since (\ref{4.19}) is a decoupled system, the topological degree of the system equals the product of the degree of the two equations in \eqref{4.19}. By Poincare-Hopf Theorem, the degree of the second equation in \eqref{4.19} is \begin{equation}
\label{4.20}
\chi(M)-|S_1\cup S_2|.
\end{equation}
By Theorem B and $Q_0\notin S_1\cup S_2$, the degree of the first equation in \eqref{4.19} has the following generating function,
\begin{equation}
\label{4.21}
(1+x+\cdots)^{1-\chi(M)}(1+\cdots x^{-K_{21}})\prod_{p\in S_2}(1+\cdots+x^{\alpha_{p,2}}).
\end{equation}
By (\ref{4.20}) and (\ref{4.21}), we can get that $d_{s,j}(0)$  has the generating function  \begin{equation}
\label{4.22}
[\chi(M)-|S_1\cup S_2|](1+x+\cdots)^{1-\chi(M)}(1+\cdots+x^{-K_{21}}) \prod_{p\in S_2}(1+\cdots+x^{\alpha_{p,2}}).
\end{equation}
Moreover, in view of Proposition  \ref{prop4.01}, we have
\begin{equation}\begin{aligned}\label{4.23}
d_{s,j}(0)&=\lim_{t\to1^-}d_{s,j}(t).\end{aligned}\end{equation}
For $\rho_2\in(4\pi j, 4\pi (j+1))$,
   let $d_{\Gamma_{w_Q,t},j}$ denote  the degree contributed by the solutions of $(1.10)_t$ from the set $\Gamma_{w_Q,t}$.
We recall that $d^S_j$ denotes the topological  degree of the  shadow system \eqref{1.1} for $\rho_2\in(4\pi j, 4\pi (j+1))$.
From (C2), \eqref{4.23},   and Lemma \ref{le4.2}-\ref{le4.3.1}, we get that \begin{equation}\begin{aligned}\label{4.26}
d^S_j=\lim_{t\to1^-}(d_{s,j}(t)-\sum_{Q\in S_2\setminus S_1}\sum_{w_Q}d_{\Gamma_{w_Q,t},j})=d_{s,j}(0)+\sum_{Q\in S_2\setminus S_1 }d_{j}(Q).\end{aligned}\end{equation}
   In view of Theorem B, the topological degree $d_{j}(Q)$ for $Q\in S_2\setminus S_1$  of the equation \eqref{4.3} has   the following generating function
   \begin{equation} (1-x)^{\chi(M)-1}(1+x+\cdots+ x^{-K_{21}+\alpha_{Q,2}}) \prod_{p\in S_2\setminus\{Q\}}(1+x+\cdots+x^{\alpha_{p,2}}).\label{4.27}\end{equation}
By using \eqref{4.22}, \eqref{4.26} and \eqref{4.27}, we conclude that   the  generating function for $d^S_j$ has the following representation:
\begin{equation*} \begin{aligned}(1-x)^{\chi(M)-1}&\Big\{(\chi(M)-|S_1\cup S_2|)(1+x+\cdots+x^{-K_{21}}) \prod_{p\in S_2}(1+x+\cdots+x^{\alpha_{p,2}})\\ &+\sum_{Q\in S_2\setminus S_1}(1+x+\cdots+ x^{-K_{21}+\alpha_{Q,2}}) \prod_{p\in S_2\setminus\{Q\}}(1+x+\cdots+x^{\alpha_{p,2}})\Big\}. \end{aligned}\end{equation*}
Hence  we finish the proof of Theorem \ref{th1.2}. \hfill $\Box$

\vspace{1cm}

\section{Applications of the degree formula of shadow system}
In the previous section, we have computed the topological degree $d^S_j$ of the shadow system (\ref{1.1}) when $\rho_2\in(4j\pi,4(j+1)\pi)$. We will use it to compute the gap $d_{1,j}^{\textbf{K}}-d_{0,j}^{\textbf{K}}$, where
$d_{i,j}^{\mathbf{K}}$  denotes  the topological degree  for $(1.14)$ when $\rho_1\in(4i\pi,4(i+1)\pi)$ and $\rho_2\in(4j\pi,4(j+1)\pi)$.

\medskip
 \noindent \textbf{Proof of   Theorem \ref{th1.4}.}
From  $d_{0,j}^{\textbf{K}}=d_{j}$ and Theorem B,  $d_{0,j}^{\textbf{K}}$ has the following generating function:
\begin{equation} \begin{aligned}(1-x)^{\chi(M)-1}   \prod_{p\in S_2}(1+x+\cdots+x^{\alpha_{p,2}}). \label{5.2}\end{aligned}\end{equation}
From Theorem F, Theorem \ref{th1.2}, and \eqref{5.2},   we see that  $d_{1,j}^{\textbf{K}}$  for  $\rho_2\in(4\pi j,4\pi({j+1}))$ has the following generating function:
\begin{equation*} \begin{aligned}(1-x)^{\chi(M)-1}&\Big\{\prod_{p\in S_2}(1+x+\cdots+x^{\alpha_{p,2}})\\&-(\chi(M)-|S_1\cup S_2|)(1+x+\cdots+x^{-K_{21}}) \prod_{p\in S_2}(1+x+\cdots+x^{\alpha_{p,2}})\\ &-\sum_{q\in S_2\setminus S_1}(1+x+\cdots+ x^{-K_{21}+\alpha_{q,2}}) \prod_{p\in S_2\setminus\{q\}}(1+x+\cdots+x^{\alpha_{p,2}})\Big\}. \end{aligned}\end{equation*}
 We can get the similar result for  $\rho_2\in(0,4\pi)\cup(4\pi,8\pi)$ and $\rho_1\in(4\pi j,4\pi(j+1))$.  Thus   we get Theorem  \ref{th1.4}.   \hfill$\square$

 \medskip

 Now we want to apply Theorem \ref{th1.4} to the equation $(1.18)$ on $M=\mathbb{S}^2$ with $n=2$ and $\mathbf{K}=\mathbf{A}_2$, that is,  \begin{equation*}
\left\{\begin{array}{l}
\Delta u^*_1+2 e^{u_1^*} -   e^{u_2^*}
=4\pi+ 4\pi\sum_{p\in S}\alpha_{p,1} \delta_p ,\\
\Delta u^*_2+2 e^{u_2^*} -   e^{u_1^*}
=4\pi+ 4\pi\sum_{p\in S}\alpha_{p,2}\delta_p,
\end{array}\right.
\end{equation*} We  recall that
 $N_1=\sum_{p\in S_1}\alpha_{p,1}$ and $N_2=\sum_{p\in S_2}\alpha_{p,2},$
where $S_1=S_2=S.$
As we discussed in the introduction,  $(1.18)$ can be written as the  form $(1.13)$ with $
(\rho_1,\rho_2)=4\pi\Big(1+\frac{2 N_1}{3}+\frac{N_2}{3}, 1+\frac{2 N_2}{3}+\frac{N_1}{3}\Big)$, that is,
\begin{equation}\label{5.3}
\left\{\begin{array}{l}
\Delta u^*_1+2\rho_1\left(\frac{ e^{u_1^*}}{\int_M e^{u_1^*}}-1\right)-\rho_2\left(\frac{ e^{u_2^*}}{\int_M e^{u_2^*}}-1\right)=4\pi\sum_{p\in S}\alpha_{p,1}(\delta_p-1),\\
\Delta u^*_2+2\rho_2\left(\frac{ e^{u_2^*}}{\int_M e^{u_2^*}}-1\right)-\rho_1\left(\frac{ e^{u_1^*}}{\int_M e^{u_1^*}}-1\right)=4\pi\sum_{p\in S}\alpha_{p,2}(\delta_p-1).
\end{array}\right.
\end{equation}
Now we are going to prove the Corollary \ref{cr1.3}.

 \medskip
 \noindent \textbf{Proof of   Corollary \ref{cr1.3}.} We note that   $\chi(\mathbb{S}^2)=2$.
Then Theorem \ref{th1.4} implies that \eqref{5.3} has the following generating function of the topological degree $d_{1,j}^{\mathbf{A}_2}$ for $(\rho_1,\rho_2)\in (4\pi,8\pi)\times (4\pi j ,4\pi (j+1))$:
\begin{equation*} \begin{aligned}\sum_{j=0}^\infty d_{1,j}^{\mathbf{A}_2} x^j=(1-x)&\Big\{\prod_{p\in S_2}(1+x+\cdots+x^{\alpha_{p,2}})\\&-(2-|S_1\cup S_2|)(1+x) \prod_{p\in S_2}(1 +\cdots+x^{\alpha_{p,2}})\\ &-\sum_{q\in S_2\setminus S_1}(1 +\cdots+ x^{1+\alpha_{q,2}}) \prod_{p\in S_2\setminus\{q\}}(1 +\cdots+x^{\alpha_{p,2}})\Big\}. \end{aligned}\end{equation*}
We consider the following several cases:

(i) if $(N_1,N_2)=(0,1)$ and $\alpha_{p,2}=1$, then $(\rho_1,\rho_2)=4\pi\Big( \frac{4}{3},  \frac{5}{3} \Big)$. We can get
 $d_{1,1}^{\mathbf{A}_2}=-1$.

(ii) if $(N_1,N_2)=(0,2)$ and $\alpha_{p,2}=1$ for any $p\in S_2$, then  $(\rho_1,\rho_2)=4\pi\Big( \frac{5}{3},  \frac{7}{3} \Big)$.
We can get  $d_{1,2}^{\mathbf{A}_2}=-1$.

(iii) if $(N_1,N_2)=(0,2)$ and $\alpha_{p,2}=2$, then $(\rho_1,\rho_2)=4\pi\Big( \frac{5}{3},  \frac{7}{3} \Big)$. We can get
 $d_{1,2}^{\mathbf{A}_2}=0$.

%(iv) if $(N_1,N_2)=(1,0)$ and $\alpha_{p,i}=1$,  then $(\rho_1,\rho_2)=4\pi\Big( \frac{5 }{3},  \frac{4}{3}\Big)$, $d_{1,1}^{\mathbf{A}_2}=-1$.

When (i) or  (ii)  holds,  we note that the degree does not vanish. As a result, we get  the existence of solutions of \eqref{5.3} and complete the proof of the Corollary \ref{cr1.3}.\hfill$\Box$

\vspace{1cm}

\end{document}